\documentclass{amsart}
\usepackage{amsmath, amssymb, amsthm, mathrsfs}
\usepackage{graphicx}

\theoremstyle{plain} 
 \newtheorem{thm}{Theorem}[section]
 \newtheorem{lem}[thm]{Lemma}
 \newtheorem{cor}[thm]{Corollary}
 \newtheorem{prop}[thm]{Proposition}
 \newtheorem{claim}[thm]{Claim}

\theoremstyle{definition}
  \newtheorem{defn}[thm]{Definition}

\theoremstyle{remark}
  \newtheorem{rem}[thm]{Remark}

\newcommand{\comm}{{\rm Comm}}
\renewcommand{\mod}{{\rm Mod}}
\renewcommand{\pmod}{{\rm PMod}}

\newcommand{\cal}{\mathcal}
\newcommand{\ci}[2]{\cite[#1]{#2}}

\newcommand{\cali}{\mathcal{I}}
\newcommand{\calk}{\mathcal{K}}
\newcommand{\calc}{\mathcal{C}}
\newcommand{\calt}{\mathcal{T}}
\newcommand{\lk}{{\rm Lk}}

\begin{document}

\title[Injections of the complex of separating curves]{Injections of the complex of separating curves into the Torelli complex}
\author{Yoshikata Kida}
\address{Department of Mathematics, Kyoto University, 606-8502 Kyoto, Japan}
\email{kida@math.kyoto-u.ac.jp}
\date{November 13, 2009, revised on February 6, 2011.}

\begin{abstract}
We show that for all but finitely many compact orientable surfaces, any superinjective map from the complex of separating curves into the Torelli complex is induced by an element of the extended mapping class group.
As an application, we prove that any injective homomorphism from a finite index subgroup of the Johnson kernel into the Torelli group for such a surface is induced by an element of the extended mapping class group.
\end{abstract}

\maketitle


\section{Introduction}

Let $S=S_{g, p}$ denote a connected, compact and orientable surface of genus $g$ with $p$ boundary components.
Throughout the paper, we assume that a surface satisfies these conditions unless otherwise stated.
Let $\chi(S)=-2g-p+2$ denote the Euler characteristic of $S$.
We define $\mod^*(S)$ as the {\it extended mapping class group} for $S$, i.e., the group of isotopy classes of homeomorphisms from $S$ onto itself, where isotopy may move points of the boundary of $S$.
The complex of curves for $S$, denoted by $\calc(S)$, plays an important role in the study of various aspects of $\mod^*(S)$.
In fact, a description of any automorphism of $\calc(S)$ leads to a description of any isomorphism between finite index subgroups of $\mod^*(S)$, as discussed by Ivanov \cite{iva-aut} and Korkmaz \cite{kork-aut} (see also \cite{luo} for automorphisms of $\calc(S)$).
More generally, Irmak \cite{irmak1} introduces superinjectivity of a simplicial map from $\calc(S)$ into itself, which is stronger than injectivity, to describe any injective homomorphism from a finite index subgroup of $\mod^*(S)$ into $\mod^*(S)$.
We refer to \cite{be-m}, \cite{bm-ar}, \cite{irmak1}, \cite{irmak2}, \cite{irmak-ns} and \cite{sha} for results of this direction.

The {\it Jonson kernel} $\calk(S)$ for $S$ is defined as the subgroup of $\mod^*(S)$ generated by all Dehn twists about separating curves in $S$.
The {\it Torelli group} $\cali(S)$ for $S$ is defined as the subgroup of $\mod^*(S)$ generated by all Dehn twists about separating curves in $S$ and all elements of the form $t_at_b^{-1}$ with $\{ a, b\}$ a bounding pair in $S$, where $t_c$ denotes the Dehn twist about a simple closed curve $c$ in $S$ (see Section \ref{sec-pre} for a precise definition of these terms).
These groups are normal subgroups of $\mod^*(S)$ and attract much attention in the study of mapping class groups. 
The complex of separating curves for $S$, denoted by $\calc_s(S)$, and the Torelli complex $\calt(S)$ for $S$ are introduced to prove algebraic results for $\calk(S)$ and $\cali(S)$ similar to those for $\mod^*(S)$ mentioned above.
We refer to \cite{bm}, \cite{bm-add}, \cite{farb-ivanov}, \cite{kida} and \cite{ky-tor} for automorphisms of $\calc_s(S)$ and $\calt(S)$.
In these works, for a certain surface $S$, any isomorphism between finite index subgroups of $\cali(S)$ (resp.\ $\calk(S)$) is shown to be the conjugation by an element of $\mod^*(S)$.

The aim of this paper is to show that any injective homomorphism from a finite index subgroup of $\calk(S)$ into $\cali(S)$ is the conjugation by an element of $\mod^*(S)$.
This result is obtained by describing any superinjective map from $\calc_s(S)$ into $\calt(S)$.
We note that $\calc_s(S)$ is a subcomplex of $\calt(S)$ and that $\mod^*(S)$ naturally acts on $\calc_s(S)$ and on $\calt(S)$ (see Section \ref{sec-pre} for a definition of these complexes).
The present paper is devoted to proving the following:

\begin{thm}\label{thm-inj}
Let $S=S_{g, p}$ be a surface with $g\geq 1$ and $|\chi(S)|\geq 4$. 
If $\phi \colon \calc_s(S)\rightarrow \calt(S)$ is a superinjective map, then the inclusion $\phi(\calc_s(S))\subset \calc_s(S)$ holds.
\end{thm}

Superinjective maps from $\calc_s(S)$ into itself are fully studied in \cite{kida-cohop}. Combining a result in \cite{kida-cohop}, stated in Theorem \ref{thm-cs} below, we obtain the following:

\begin{cor}\label{cor-inj}
Let $S$ be the surface in Theorem \ref{thm-inj}. 
For any superinjective map $\phi \colon \calc_s(S)\rightarrow \calt(S)$, there exists an element $\gamma$ of $\mod^*(S)$ with the equality $\phi(a)=\gamma a$ for any vertex $a$ of $\calc_s(S)$.
\end{cor}

If $S$ is a surface of genus zero, then both $\calc_s(S)$ and $\calt(S)$ are equal to $\calc(S)$. 
The same conclusion as Corollary \ref{cor-inj} is thus proved in \cite{bm-ar}.
It remains open whether the same conclusions as Theorem \ref{thm-inj} and Corollary \ref{cor-inj} are true for $S=S_{1, 3}$. 
We refer to \cite{kida}, \cite{kida-cohop} and \cite{ky-tor} for known facts on the complexes $\calc_s(S)$ and $\calt(S)$ for a surface $S$ with $|\chi(S)|=3$ (see also Remark \ref{rem-13}).

An idea to prove Theorem \ref{thm-inj} for closed surfaces has already appeared in \cite{bm}. 
The proof for surfaces with non-empty boundary is much harder because of the existence of bounding pairs cutting off a pair of pants. 
A sketch of the proof is given in Section \ref{sec-str}.

In the rest of this section, we state results on the groups $\calk(S)$ and $\cali(S)$ obtained as a consequence of Corollary \ref{cor-inj}.

\begin{thm}\label{thm-inner}
Let $S$ be the surface in Theorem \ref{thm-inj}. 
Let $\Gamma$ be a finite index subgroup of $\calk(S)$ and $f\colon \Gamma \rightarrow \cali(S)$ an injective homomorphism. 
Then there exists a unique element $\gamma_0$ of $\mod^*(S)$ with the equality $f(\gamma)= \gamma_0\gamma \gamma_0^{-1}$ for any $\gamma \in \Gamma$.
\end{thm}

This theorem is deduced by combining Proposition 6.9 in \cite{kida} with Corollary \ref{cor-inj} and following the proof of Theorem 1.4 in \cite{kida}.

To state consequences of Theorem \ref{thm-inner}, let us introduce terminology.
For a group $G$ and its subgroup $\Gamma$, the {\it relative commensurator} of $\Gamma$ in $G$, denoted by $\comm_G(\Gamma)$, is defined as the subgroup of $G$ consisting of all elements $\gamma \in G$ such that $\Gamma \cap \gamma \Gamma \gamma^{-1}$ is of finite index in both $\Gamma$ and $\gamma \Gamma \gamma^{-1}$.

For a group $\Gamma$, we define $F(\Gamma)$ to be the set of isomorphisms between finite index subgroups of $\Gamma$. 
We say that two elements $f$, $h$ of $F(\Gamma)$ are equivalent if there exists a finite index subgroup of $\Gamma$ on which $f$ and $h$ are equal. 
The composition of two elements $f\colon \Gamma_1\rightarrow \Gamma_2$, $h\colon \Lambda_1\rightarrow \Lambda_2$ of $F(\Gamma)$ given by $f\circ h\colon h^{-1}(\Gamma_1\cap \Lambda_2)\rightarrow f(\Lambda_2\cap \Gamma_1)$ induces the product operation on the quotient set of $F(\Gamma)$ by this equivalence relation. 
This makes it into the group called the {\it abstract commensurator} of $\Gamma$ and denoted by $\comm(\Gamma)$. 
If $\Gamma$ is a subgroup of a group $G$, then we have the natural homomorphism from $\comm_G(\Gamma)$ into $\comm(\Gamma)$.

\begin{thm}\label{thm-comm}
Let $S$ be the surface in Theorem \ref{thm-inj} and put $G=\mod^*(S)$.
Let $\Gamma$ be a subgroup of $G$ with $[\calk(S): \Gamma \cap \calk(S)]<\infty$ and $[\Gamma :\Gamma \cap \cali(S)]<\infty$.
Then the natural homomorphism from $\comm_G(\Gamma)$ into $\comm(\Gamma)$ is an isomorphism.
\end{thm}

Recall that a group $\Gamma$ is said to be {\it co-Hopfian} if any injective homomorphism from $\Gamma$ into itself is surjective.

\begin{thm}\label{thm-coh}
Let $S=S_{g, p}$ be a surface and assume either $g\geq 3$ and $p\leq 1$ or $g=1$ and $p\geq 4$.
Then any subgroup $\Gamma$ of $\mod^*(S)$ with $[\calk(S): \Gamma \cap \calk(S)]<\infty$ and $[\Gamma :\Gamma \cap \cali(S)]<\infty$ is co-Hopfian.
In particular, any intermediate subgroup between $\calk(S)$ and $\cali(S)$ is co-Hopfian.
\end{thm}

The paper is organized as follows.
In Section \ref{sec-pre}, we review a definition of the complexes $\calc(S)$, $\calc_s(S)$ and $\calt(S)$ and the groups $\calk(S)$ and $\cali(S)$. 
We also recall known facts on simplicial maps between those complexes.
In Section \ref{sec-str}, we present an outline of the proof of Theorem \ref{thm-inj}.
In Section \ref{sec-simp}, we introduce several simplicial graphs associated to surfaces, which will be used in subsequent sections.
We provide basic properties of them.
The proof of Theorem \ref{thm-inj} is given throughout Sections \ref{sec-chi}--\ref{sec-other}.
Details of the organization of the proof is explained in Section \ref{sec-str}.
In Section \ref{sec-coh}, we prove Theorems \ref{thm-comm} and \ref{thm-coh}.


\section{Preliminaries}\label{sec-pre}

\subsection{Terminology}

Unless otherwise stated, we assume that a surface is connected, compact and orientable. 
Let $S=S_{g, p}$ denote a surface of genus $g$ with $p$ boundary components. 
A simple closed curve in $S$ is said to be {\it essential} in $S$ if it is neither homotopic to a point of $S$ nor isotopic to a boundary component of $S$. 
When there is no confusion, we mean by a curve in $S$ either an essential simple closed curve in $S$ or the isotopy class of it. 
A curve $a$ in $S$ is said to be {\it separating} in $S$ if $S\setminus a$ is not connected.
Otherwise, $a$ is said to be {\it non-separating} in $S$.
These properties depend only on the isotopy class of $a$.
A pair of non-separating curves in $S$, $\{ a, b \}$, is called a {\it bounding pair (BP)} in $S$ if $a$ and $b$ are disjoint and non-isotopic and if $S\setminus (a \cup b)$ is not connected.
These conditions depend only on the isotopy classes of $a$ and $b$.
We say that two non-separating curves in $S$ are {\it BP-equivalent} in $S$ if they either are isotopic or are disjoint and form a BP in $S$.

A surface homeomorphic to $S_{1, 1}$ is called a {\it handle}.
A surface homeomorphic to $S_{0, 3}$ is called a {\it pair of pants}.
Let $a$ be a separating curve in $S$. 
If $a$ cuts off a handle from $S$, then $a$ is called an {\it h-curve} in $S$. 
If $a$ cuts off a pair of pants from $S$, then $a$ is called a {\it p-curve} in $S$. 
When $S$ is not homeomorphic to $S_{0, 4}$, for a p-curve $a$ in $S$ and a component $\partial$ of $\partial S$, we say that $a$ {\it cuts off} $\partial$ if the pair of pants cut off by $a$ from $S$ contains $\partial$. 
A curve in $S$ which is either an h-curve in $S$ or a p-curve in $S$ is called an {\it hp-curve} in $S$. 
If a BP $b$ in $S$ cuts off a pair of pants from $S$, then $b$ is called a {\it p-BP} in $S$. 
When $S$ is not homeomorphic to $S_{1, 2}$, for a p-BP $b$ in $S$ and a component $\partial$ of $\partial S$, we say that $b$ {\it cuts off} $\partial$ if the pair of pants cut off by $b$ from $S$ contains $\partial$.

We denote by $V(S)$ the set of isotopy classes of essential simple closed curves in $S$.
For two elements $\alpha, \beta\in V(S)$, we define $i(\alpha, \beta)$ the {\it geometric intersection number}, i.e., the minimal cardinality of $A\cap B$ among representatives $A$ and $B$ of $\alpha$ and $\beta$, respectively.
Let $\Sigma(S)$ denote the set of non-empty finite subsets $\sigma$ of $V(S)$ with $i(\alpha, \beta)=0$ for any $\alpha, \beta \in \sigma$.
We mean by a {\it representative} of an element $\sigma$ of $\Sigma(S)$ the union of mutually disjoint representatives of elements in $\sigma$.

We extend the function $i$ to the symmetric function on $(V(S)\sqcup \Sigma(S))^2$ so that we have $i(\alpha, \sigma)=\sum_{\beta \in \sigma}i(\alpha, \beta)$ and $i(\sigma, \tau)=\sum_{\beta \in \sigma, \gamma \in \tau}i(\beta, \gamma)$ for any $\alpha \in V(S)$ and $\sigma, \tau \in \Sigma(S)$.
Let $\alpha$ and $\beta$ be elements of $V(S)\sqcup \Sigma(S)$.
We say that $\alpha$ and $\beta$ are {\it disjoint} if $i(\alpha, \beta)=0$, and otherwise they {\it intersect}.
For representatives $A$, $B$ of $\alpha$, $\beta$, respectively, we say that $A$ and $B$ {\it intersect minimally} if $|A\cap B|=i(\alpha, \beta)$.

For an element $\alpha$ of $V(S)$ (or its representative), we denote by $S_{\alpha}$ the surface obtained by cutting $S$ along $\alpha$. 
Similarly, for an element  $\sigma$ of $\Sigma(S)$ (or its representative), we denote by $S_{\sigma}$ the surface obtained by cutting $S$ along all curves in $\sigma$.
Each component of $S_{\sigma}$ is often identified with a complementary component of a tubular neighborhood of a one-dimensional submanifold representing $\sigma$ in $S$.
For each component $Q$ of $S_{\sigma}$, the set $V(Q)$ is then identified with a subset of $V(S)$.

Suppose that the boundary $\partial S$ of $S$ is non-empty.
We say that a simple arc $l$ in $S$ is {\it essential} in $S$ if
\begin{itemize}
\item $\partial l$ consists of two distinct points of $\partial S$;
\item $l$ meets $\partial S$ only at its end points; and
\item $l$ is not isotopic relative to $\partial l$ to an arc in $\partial S$.
\end{itemize}
Unless otherwise stated, isotopy of essential simple arcs in $S$ may move their end points, keeping them staying in $\partial S$.
An essential simple arc $l$ in $S$ is said to be {\it separating} in $S$ if $S\setminus l$ is not connected.
Otherwise, $l$ is said to be {\it non-separating} in $S$.
These properties depend only on the isotopy class of $l$.

Let $\partial_1$ and $\partial_2$ be distinct components of $\partial S$.
We say that an essential simple arc $l$ in $S$ {\it connects $\partial_1$ and $\partial_2$} if one of the end point of $l$ lies in $\partial_1$ and another in $\partial_2$.
In this case, we obtain the p-curve $a$ in $S$ that is the boundary component of a regular neighborhood of $l\cup \partial_1\cup \partial_2$ in $S$ other than $\partial_1$ and $\partial_2$.
The isotopy class of $a$ depends only on the isotopy class of $l$.
The curve $a$ is then called the curve in $S$ {\it defined by} $l$.
Conversely, if $b$ is a p-curve in $S$ cutting off $\partial_1$ and $\partial_2$, then up to isotopy, there exists a unique essential simple arc in $S$ connecting $\partial_1$ and $\partial_2$ and disjoint from a curve in $S$ isotopic to $b$.
Such an arc is called a {\it defining arc} of (the isotopy class of) $b$.


\subsection{Simplicial complexes associated to surfaces}\label{subsec-comp}

We fix a surface $S$.
We recall three simplicial complexes associated to $S$.
The complex of curves is originally introduced by Harvey \cite{harvey}.
The complex of separating curves appears in \cite{bm}, \cite{bm-add}, \cite{farb-ivanov} and \cite{mv}.
The Torelli complex (with a certain marking and for a closed surface) is originally introduced by Farb-Ivanov \cite{farb-ivanov}.

\medskip

\noindent {\bf The complex of curves.} We define $\calc(S)$ as the abstract simplicial complex such that the sets of vertices and simplices of $\calc(S)$ are $V(S)$ and $\Sigma(S)$, respectively.
The complex $\calc(S)$ is called the {\it complex of curves} for $S$.

\medskip

\noindent {\bf The complex of separating curves.} Let $V_s(S)$ denote the set of all elements of $V(S)$ whose representatives are separating in $S$.
We define $\calc_s(S)$ as the full subcomplex of $\calc(S)$ spanned by $V_s(S)$ and call it the {\it complex of separating curves} for $S$. 

\medskip

\noindent {\bf The Torelli complex.} Let $V_{bp}(S)$ denote the set of all elements of $\Sigma(S)$ whose representatives are BPs in $S$.
The {\it Torelli complex} for $S$, denoted by $\calt(S)$, is defined as the abstract simplicial complex such that the set of vertices of $\calt(S)$ is the disjoint union $V_s(S)\sqcup V_{bp}(S)$, and a non-empty finite subset $\sigma$ of $V_s(S)\sqcup V_{bp}(S)$ is a simplex of $\calt(S)$ if and only if any two elements of $\sigma$ are disjoint.
Let $V_t(S)$ denote the set of vertices of $\calt(S)$.

\medskip

For a simplex $\sigma$ of $\calt(S)$, we denote by $S_{\sigma}$ the surface obtained by cutting $S$ along all separating curves in $\sigma$ and all non-separating curves in BPs in $\sigma$, where isotopic curves are identified. 
We say that two elements $b_1$, $b_2$ of $\sigma \cap V_{bp}(S)$ are {\it BP-equivalent} if any two distinct curves of $b_1\cup b_2$ form a BP in $S$. 
A {\it BP-equivalence class} of $\sigma$ is an equivalence class of $\sigma \cap V_{bp}(S)$ with respect to the BP-equivalence relation. 
The following lemma is a basic observation on BPs, which will be used many times throughout the paper.

\begin{lem}[\ci{Lemma 3.1}{kida}]\label{lem-bp}
Let $a$ be a BP in $S$, and let $b$ be either a separating curve in $S$ with $i(a, b)=0$ or a BP in $S$ which satisfies $i(a, b)=0$, but is not BP-equivalent to $a$. 
Then the two curves in $a$ are in a single component of $S_b$.
\end{lem}

\noindent {\bf Superinjective maps.} Let $X$ and $Y$ be any of the simplicial complexes $\calc(S)$, $\calc_s(S)$ and $\calt(S)$. 
We denote by $V(X)$ and $V(Y)$ the sets of vertices of $X$ and $Y$, respectively. Note that a map $\phi \colon V(X)\rightarrow V(Y)$ defines a simplicial map from $X$ into $Y$ if and only if $i(\phi(a), \phi(b))=0$ for any two vertices $a, b\in V(X)$ with $i(a, b)=0$. 
We mean by a {\it superinjective map} $\phi \colon X\rightarrow Y$ a simplicial map $\phi \colon X\rightarrow Y$ satisfying $i(\phi(a), \phi(b))\neq 0$ for any two vertices $a, b\in V(X)$ with $i(a, b)\neq 0$. 
One can check that any superinjective map is injective (see Section 2.2 of \cite{kida}).

Let $g$ denote the genus of $S$ and $p$ the number of boundary components of $S$.
We assume $3g+p-4>0$. 
As proved in \cite{iva-aut}, \cite{kork-aut} and \cite{luo}, any automorphism of $\calc(S)$ is generally induced by an element of $\mod^*(S)$. 
In \cite{be-m}, \cite{bm-ar}, \cite{irmak1}, \cite{irmak2} and \cite{irmak-ns}, it is shown that any superinjective map from $\calc(S)$ into itself is surjective. 
More generally, in \cite{sha}, any injective simplicial map from $\calc(S)$ into itself is shown to be surjective. 
As for the complex of separating curves, we know the following:

\begin{thm}[\cite{kida-cohop}]\label{thm-cs}
Let $S=S_{g, p}$ be a surface and assume one of the following three conditions: $g\geq 3$ and $p\geq 0$; $g=2$ and $p\geq 2$; or $g=1$ and $p\geq 3$. 
Then any superinjective map from $\calc_s(S)$ into itself is induced by an element of $\mod^*(S)$.
\end{thm}

It is also known that if $S$ is the surface in Theorem \ref{thm-cs}, then any superinjective map from $\calt(S)$ into itself is induced by an element of $\mod^*(S)$ (see \cite{kida-cohop}).


\subsection{The Johnson kernel and the Torelli group}\label{subsec-mcg}

Let $S$ be a surface. 
We define $\mod(S)$ as the {\it mapping class group} for $S$, i.e., the subgroup of $\mod^*(S)$ consisting of all isotopy classes of orientation-preserving homeomorphisms from $S$ onto itself.
We define $\pmod(S)$ as the {\it pure mapping class group} for $S$, i.e., the subgroup of $\mod^*(S)$ consisting of all isotopy classes of homeomorphisms from $S$ onto itself that preserve an orientation of $S$ and each boundary component of $S$ as a set.
Both $\mod(S)$ and $\pmod(S)$ are normal subgroups of $\mod^*(S)$ of finite index.

For each $a\in V(S)$, we denote by $t_a\in \pmod(S)$ the {\it (left) Dehn twist} about $a$. 
The {\it Johnson kernel} $\calk(S)$ is the subgroup of $\pmod(S)$ generated by all $t_a$ with $a\in V_s(S)$. 
The {\it Torelli group} $\cali(S)$ for $S$ is the subgroup of $\pmod(S)$ generated by all $t_a$ with $a\in V_s(S)$ and all $t_b t_c^{-1}$ with $\{ b, c\} \in V_{bp}(S)$.
Both $\calk(S)$ and $\cali(S)$ are normal subgroups of $\mod^*(S)$ because the equality $\gamma t_a\gamma^{-1}=t_{\gamma a}^{\varepsilon}$ holds for any $a\in V(S)$ and $\gamma \in \mod^*(S)$, where $\varepsilon =1$ if $\gamma \in \mod(S)$, and $\varepsilon =-1$ otherwise.


\section{Strategy}\label{sec-str}

Let $S=S_{g, p}$ be a surface with $g\geq 1$ and $|\chi(S)|\geq 4$.
This is equivalent to the condition that we have either $g\geq 3$ and $p\geq 0$; $g=2$ and $p\geq 2$; or $g=1$ and $p\geq 4$.
Let $\phi \colon \calc_s(S)\rightarrow \calt(S)$ be a superinjective map.
In what follows, we present an outline to prove the inclusion $\phi(\calc_s(S))\subset \calc_s(S)$.

\medskip

\noindent {\bf Step 1.} In Section \ref{sec-chi}, we prove that $\phi$ is $\chi$-preserving. 
Namely, for each $\alpha \in V_s(S)$, if we denote by $Q_1$ and $Q_2$ the two components of $S_{\alpha}$ and denote by $R_1$ and $R_2$ the two components of $S_{\phi(\alpha)}$, then the equality $\chi(Q_j)=\chi(R_j)$ holds for each $j=1, 2$ after exchanging the indices if necessary. 
In particular, $\phi$ sends each hp-curve in $S$ to either an hp-curve in $S$ or a p-BP in $S$.

When $p=0$, $\phi$ preserves h-curves in $S$ because there exists neither p-curve in $S$ nor p-BP in $S$.
The inclusion $\phi(\calc_s(S))\subset \calc_s(S)$ then readily follows (see Section \ref{subsec-p0}).
When $p\geq 1$, the inclusion is proved by induction on $p$ in Steps 2 and 3.

\medskip

\noindent {\bf Step 2.} This is the first step of the induction on $p$.
We deal with $S_{g, 1}$ with $g\geq 3$ and $S_{2, 2}$ in Section \ref{sec-p1} and deal with $S_{1, 4}$ in Section \ref{sec-14}.

When either $S=S_{g, 1}$ with $g\geq 3$ or $S=S_{2, 2}$, most of the proof of the inclusion $\phi(\calc_s(S))\subset \calc_s(S)$ is devoted to deducing a contradiction on the assumption that we have an h-curve $a$ in $S$ with $\phi(a)$ a p-BP in $S$.
A central idea to deduce a contradiction is based on the implicit fact that there are curves in $S_{2, 1}$ more than in $S_{1, 3}$ in spite of the equality $\chi(S_{2, 1})=\chi(S_{1, 3})$.
Indeed, $S_{1, 3}$ is embedded in $S_{2, 1}$ as the complement of a tubular neighborhood of a non-separating curve in $S_{2, 1}$.
We fix a separating curve $z$ in $S$ such that $z$ is disjoint from $a$ and non-isotopic to $a$; and the component of $S_z$ containing $a$, denoted by $X$, is homeomorphic to $S_{2, 1}$.
We can then show that $\phi(z)$ is a BP in $S$ containing a curve in $\phi(a)$.
Since $\phi$ is $\chi$-preserving, the component of $S_{\phi(z)}$, denoted by $Y$, containing a curve in $\phi(a)$ as an essential one is homeomorphic to $S_{1, 3}$.
We now pick an h-curve $b$ in $X$ such that $\{ a, b\}$ is a sharing pair in $X$ (see Definition \ref{defn-share}).
The image $\phi(b)$ is understood by using the curves in Figure \ref{fig-s-pair-cha} that are originally introduced by Brendle-Margalit \cite{bm-add} to characterize sharing pairs in terms of disjointness and non-disjointness.
To deduce a contradiction, we introduce the simplicial graph $\cal{B}$ whose vertices are h-curves in $X$ forming a sharing pair in $X$ together with $a$ and satisfying an additional condition.
We also introduce the simplicial graph $\cal{G}$ whose vertices are certain p-curves in $Y$.
Using information on $\phi(b)$, we construct an injective simplicial map from $\cal{B}$ into $\cal{G}$.
On the other hand, in Section \ref{subsec-inj}, it is shown that there exists no injective simplicial map from $\cal{B}$ into $\cal{G}$.
We thus obtain a contradiction.

When $S=S_{1, 4}$, on the assumption that we have an hp-curve $a$ in $S$ with $\phi(a)$ a p-BP in $S$, we deduce a contradiction in an analogous way.
We aim to construct an injective simplicial map from a simplicial graph associated to $S_{1, 2}$ into a simplicial graph associated to $S_{0, 4}$.
The definition of these graphs and that of $\cal{B}$ and $\cal{G}$ are based on a similar idea.
Using results in Section \ref{subsec-inj}, we deduce a contradiction.

\begin{rem}\label{rem-13}
When $S=S_{1, 4}$, the choice of curves in Figure \ref{fig-14} is crucial in the construction of the injective simplicial map mentioned above.
We do not know similar choice of curves in the case $S=S_{1, 3}$. 
This is why $S_{1, 3}$ is not dealt with in Theorem \ref{thm-inj} although it is shown in \cite{kida-cohop} that for $S=S_{1, 3}$, any superinjective map from $\calc_s(S)$ into itself is induced by an element of $\mod^*(S)$.
\end{rem}

\noindent {\bf Step 3.} In Section \ref{sec-other}, we deal with the remainder of surfaces.
The proof consists of straightforward arguments using the hypothesis of the induction on $p$.


\section{Simplicial graphs associated to surfaces}\label{sec-simp}

In this section, we introduce several simplicial graphs associated to $S_{0, 4}$, $S_{1, 2}$, $S_{1, 3}$ and $S_{2, 1}$, and show results on non-existence of injective simplicial maps between those graphs.
The results will be used in Section \ref{sec-p1} and \ref{sec-14}.

\subsection{Simplicial graphs associated to $S_{0, 4}$}\label{subsec-farey}

Throughout this subsection, we put $R=S_{0, 4}$.
The aim of this subsection is to provide basic facts on the graphs $\cal{F}$, $\cal{E}$ and $\cal{H}$ introduced below.

\medskip

\noindent {\bf Graph $\cal{F}$.} Let $R=S_{0, 4}$ be a surface.
We define $\cal{F}=\cal{F}(R)$ as the simplicial graph such that the set of vertices of $\cal{F}$ is $V(R)$, and two vertices $a$, $b$ of $\cal{F}$ are connected by an edge of $\cal{F}$ if and only if we have $i(a, b)=2$.

\medskip

This graph is known to be isomorphic to the Farey graph realized as an ideal triangulation of the Poincar\'e disk (see Section 3.2 in \cite{luo}).

Let $G$ be a simplicial graph.
We mean by a {\it triangle} in $G$ a subgraph of $G$ consisting of exactly three vertices and exactly three edges. 
We say that two vertices $v_1$, $v_2$ of $G$ {\it lie in a diagonal position of two adjacent triangles in $G$} if there exist two triangles $\Delta_1$, $\Delta_2$ in $G$ such that $v_1\in \Delta_1$, $v_2\in \Delta_2$ and $\Delta_1\cap \Delta_2$ is an edge of $G$ containing neither $v_1$ nor $v_2$.

\begin{prop}\label{prop-diag}
Let $R=S_{0, 4}$ be a surface and put $\cal{F}=\cal{F}(R)$.
Let $a$ and $b$ be curves in $R$.
Then the following three conditions are equivalent:
\begin{enumerate}
\item[(a)] The two vertices $a$, $b$ of $\cal{F}$ lie in a diagonal position of two adjacent triangles in $\cal{F}$.
\item[(b)] There exist defining arcs $l$, $r$ of $a$, $b$, respectively, such that the set of the two components of $\partial R$ connected by $l$ and that for $r$ are equal; and $l$ and $r$ are disjoint and non-isotopic.
\item[(c)] We have $i(a, b)=4$.
\end{enumerate}
\end{prop}

\begin{proof}
Let $\partial_1,\ldots, \partial_4$ denote the four components of $\partial R$.
Note that $\mod(R)$ acts transitively on the set of oriented edges of $\cal{F}$.
It follows that $\pmod(R)$ acts transitively on the set of all non-oriented edges of $\cal{F}$ consisting of two curves $c_3$, $c_4$ in $R$ such that for each $j=3, 4$, $c_j$ cuts off a pair of pants containing $\partial_1$ and $\partial_j$ from $R$.
The group $\pmod(R)$ thus acts transitively on the set of all non-ordered pairs of two curves $d_1$, $d_2$ in $R$ such that each of $d_1$ and $d_2$ cuts off a pair of pants containing $\partial_1$ and $\partial_2$ from $R$; and $d_1$ and $d_2$ lie in a diagonal position of two adjacent triangles in $\cal{F}$.
\begin{figure}
\begin{center}
\includegraphics[width=6cm]{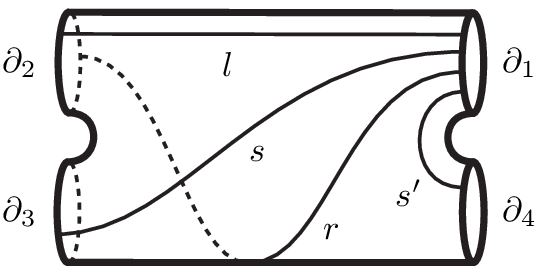}
\caption{}\label{fig-arcs}
\end{center}
\end{figure}
We can thus show that condition (a) implies conditions (b) and (c) by describing four curves $a_1$, $a_2$, $a_3$ and $a_4$ in $R$ such that each of $\{ a_1, a_3, a_4\}$ and $\{ a_2, a_3, a_4\}$ forms a triangle in $\cal{F}$.

We assume condition (b).
We may assume that $l$ and $r$ connect $\partial_1$ and $\partial_2$.
Let $h\in \mod(R)$ be the half twist about $a$ exchanging $\partial_1$ and $\partial_2$ and being the identity on the component of $R_a$ containing $\partial_3$ and $\partial_4$.
If $l$ is described as in Figure \ref{fig-arcs}, then $r$ is described as in the same Figure up to a power of $h$.
Let $s$ and $s'$ be the essential simple arcs in $R$ described in the same Figure.
Let $c$ and $c'$ denote the curves in $R$ defined by $s$ and $s'$, respectively. 
Each of $\{ a, c, c'\}$ and $\{ b, c, c'\}$ forms a triangle in $\cal{F}$. 
We thus obtain condition (a).

Finally, we prove that condition (c) implies condition (b).
Although this can readily be shown by using Lemma 3.2 in \cite{kida-cohop}, we present a direct proof without using it.
We assume $i(a, b)=4$.
We may assume that $a$ cuts off a pair of pants containing $\partial_1$ and $\partial_2$ from $R$.
Let $A$ and $B$ be representatives of $a$ and $b$, respectively, with $|A\cap B|=i(a, b)$.
We denote by $P$ the component of $R_A$ containing $\partial_3$ and $\partial_4$.
We put $P'=R_A\setminus P$.
Since we have $|A\cap B|=4$, the intersection $B\cap P$ consists of two essential simple arcs in $P$, denoted by $s_1$ and $s_2$.
The arcs $s_1$ and $s_2$ are isotopic because $P$ is a pair of pants.
The intersection $B\cap P'$ also consists of two isotopic, essential simple arcs in $P'$, denoted by $s_3$ and $s_4$.

Fix an orientation of $A$.
For each $j=1, 2$, we put $\partial s_j=\{ p_j, q_j\}$ so that $p_1$, $q_1$, $q_2$ and $p_2$ appear along $A$ in this order.
For each $j=3, 4$, the arc $s_j$ connects neither $p_1$ and $q_1$ nor $p_2$ and $q_2$ because otherwise $s_j$ and either $s_1$ or $s_2$ would form a simple closed curve in $R$.
For each $j=3, 4$, the arc $s_j$ connects neither $p_1$ and $q_2$ nor $p_2$ and $q_1$ because $s_j$ is separating in $P'$.
It turns out that $s_3$ and $s_4$ connect either $p_1$ and $p_2$ or $q_1$ and $q_2$.

Let $I$ and $J$ denote the components of $A\setminus \{ p_1, p_2\}$ and $A\setminus \{ q_1, q_2\}$, respectively, that contain no point of $A\cap B$.
We may assume that $I$ and $\partial_1$ (resp.\ $J$ and $\partial_2$) lie in the same component of $P'\setminus B$.
Pick essential simple arcs $u_1$, $u_2$ in $P'$ such that
\begin{itemize}
\item $u_1$ connects a point of $\partial_1$ with a point of $I$, and $u_2$ connects a point of $\partial_2$ with a point of $J$; and
\item both $u_1$ and $u_2$ are disjoint from $B\cap P'$.
\end{itemize}
Since $s_1\cup s_2$ cuts off from $P$ a disk whose boundary is the union of $s_1$, $s_2$, $I$ and $J$, there exists an essential simple arc $u_3$ in $P$ disjoint from $B\cap P$ and connecting the point of $u_1\cap I$ with the point of $u_2\cap J$.
We define $r$ as the union $u_1\cup u_2\cup u_3$, which is an essential simple arc in $R$ connecting $\partial_1$ and $\partial_2$ and disjoint from $B$.
Pick an essential simple arc $l$ in $P'$ connecting $\partial_1$ and $\partial_2$ and disjoint from $u_1$ and $u_2$.
The arc $l$ is disjoint from $A$ and $r$.
Condition (b) is obtained.
\end{proof}

Lemma \ref{lem-pre-h4-p4} motivates us to introduce the following:

\medskip

\noindent {\bf Graph $\cal{E}$.} Let $R=S_{0, 4}$ be a surface.
Fix two distinct components $\partial_1$, $\partial_2$ of $\partial R$. 
We define a simplicial graph $\cal{E}=\cal{E}(R; \partial_1, \partial_2)$ as follows: Vertices of $\cal{E}$ are elements $a$ of $V(R)$ such that $\partial_1$ and $\partial_2$ are contained in the same component of $R_a$. 
Two vertices $a$, $b$ of $\cal{E}$ are connected by an edge of $\cal{E}$ if and only if we have $i(a, b)=4$.

\begin{prop}\label{prop-e-tree}
In the above notation, the graph $\cal{E}$ is a tree.
\end{prop}

\begin{proof}
Let $V(\cal{E})$ denote the set of vertices of $\cal{E}$.
Let $\cal{F}$ denote the graph $\cal{F}(R)$.
By Proposition \ref{prop-diag}, for any two vertices $a$, $b$ of $\cal{E}$, $a$ and $b$ form an edge of $\cal{E}$ if and only if $a$ and $b$ lie in a diagonal position of two adjacent triangles in $\cal{F}$.
A notable property of the graph $\cal{F}$ is that for each edge $e$ of $\cal{F}$, the set $\cal{F}\setminus \bar{e}$ has exactly two connected components, where $\bar{e}$ denotes the closure $e\cup \partial e$ of $e$ in the geometric realization of $\cal{F}$.
Let us call a component of $\cal{F}\setminus \bar{e}$ a {\it side} of $e$. 
For each vertex $v$ of $\cal{F}$ and each edge $e$ of $\cal{F}$ with $v\not\in \partial e$, we denote by $X_{v, e}$ the set of all vertices of $\cal{F}$ contained in the side of $e$ that does not contain $v$. 
For each vertex $v$ of $\cal{F}$, let $E_v$ denote the set of all edges $e$ of $\cal{F}$ such that $v$ and $e$ are contained in a single triangle in $\cal{F}$ with $v\not\in \partial e$.
If $v\in V(\cal{E})$, then $V(\cal{E})\setminus \{ v\}$ is equal to the disjoint union $\bigsqcup_{e\in E_v}(X_{v, e}\cap V(\cal{E}))$.

Fix $v_0\in V(\cal{E})$.
We construct a deformation retraction of $\cal{E}$ into $v_0$.
We define a map $f\colon V(\cal{E})\rightarrow V(\cal{E})$ as follows.
Set $f(v_0)=v_0$.
For each $v\in V(\cal{E})$ with $v\neq v_0$, choose a unique edge $e$ of $E_v$ with $v_0\in X_{v, e}$.
Let $\Delta$ be the triangle in $\cal{F}$ that contains $e$, but does not contain $v$. 
We define $f(v)\in V(\cal{E})$ as the vertex of $\Delta$ that is not a vertex of $e$.

By the definition of $f$, for each $v\in V(\cal{E})$, $f(v)$ and $v$ either are equal or form an edge of $\cal{E}$.
We claim that for each edge $\{ u, v\}$ of $\cal{E}$, either $f(u)=v$ or $f(v)=u$ holds.
In particular, $f$ defines a simplicial map from $\cal{E}$ into itself.
If either $u$ or $v$ is $v_0$, then $f(u)=f(v)=v_0$.
Assume that neither $u$ nor $v$ is $v_0$.
Let $e$ be the edge of $\cal{F}$ such that $u$ and $v$ lie in distinct sides of $e$.
Note that $v_0$ is not a vertex of $e$ because no vertex of $e$ corresponds to a curve in $R$ cutting off a pair of pants containing $\partial_1$ and $\partial_2$ from $R$.
If $v_0$ and $v$ are in the same side of $e$, then $f(u)=v$.
If $v_0$ and $u$ are in the same side of $e$, then $f(v)=u$.
The claim follows.

If $v$ is a vertex of $\cal{E}$ with $v\neq v_0$, then with respect to the path metric on $\cal{F}$, the distance between $v_0$ and $f(v)$ is strictly smaller than that between $v_0$ and $v$ by the definition of $f$.
For each $v\in V(\cal{E})$, there thus exists a positive integer $n$ with $f^n(v)=v_0$.
The iteration of $f$ defines a deformation retraction of $\cal{E}$ into $v_0$.
\end{proof}

Lemma \ref{lem-pre-h2-diag} motivates us to introduce the following:

\medskip

\noindent {\bf Graph $\cal{H}$.} Let $R=S_{0, 4}$ be a surface and put $\cal{F}=\cal{F}(R)$.
We fix two distinct components $\partial_1$, $\partial_2$ of $\partial R$. 
We define the simplicial graph $\cal{H}=\cal{H}(R; \partial_1, \partial_2)$ as follows: Vertices of $\cal{H}$ are elements $a$ of $V(R)$ such that $\partial_1$ and $\partial_2$ lie in distinct components of $R_a$. 
\begin{figure}
\begin{center}
\includegraphics[width=12cm]{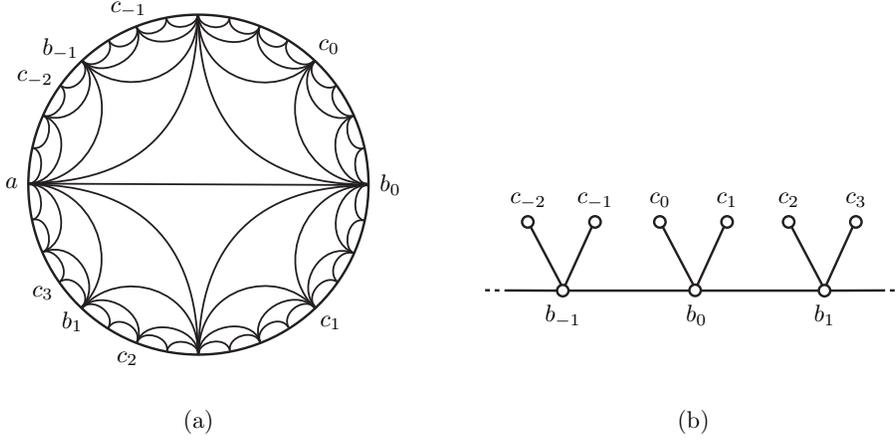}
\caption{(a) The Farey graph; (b) The link of $a$ in $\cal{H}$}\label{fig-hlink}
\end{center}
\end{figure}
Two vertices $a$, $b$ of $\cal{H}$ are connected by an edge of $\cal{H}$ if and only if either we have $i(a, b)=2$ or $a$ and $b$ lie in a diagonal position of two adjacent triangles in $\cal{F}$.

\begin{lem}
In the above notation, for each vertex $a$ of $\cal{H}$, the link of $a$ in $\cal{H}$ is the graph described in Figure \ref{fig-hlink} (b), where the vertices $b_n$ and $c_n$ are located so that for each $n\in \mathbb{Z}$,
\begin{itemize}
\item we have $i(a, b_n)=i(b_n, c_{2n})=i(b_n, c_{2n+1})=2$;
\item $a$ and $c_n$ lie in a diagonal position of two adjacent triangles in $\cal{F}$; and
\item $b_n$ and $b_{n+1}$ lie in a diagonal position of two adjacent triangles in $\cal{F}$.
\end{itemize}
\end{lem}

\begin{proof}
We first note that if $a_1$, $a_2$ and $a_3$ are vertices of $\cal{F}$ forming a triangle in $\cal{F}$, then exactly two of them are vertices of $\cal{H}$.
Let $L$ denote the link of $a$ in $\cal{H}$.
Let $P$ denote the component of $R_a$ containing $\partial_1$ as its boundary component.
We define $h\in \mod(R)$ as the half twist about $a$ exchanging the two components of $\partial R$ contained in $P$ and being the identity on $R_a\setminus P$.
If we pick a vertex $b_0$ of $\cal{H}$ with $i(a, b_0)=2$, then the set $\{ h^n(b_0)\}_{n\in \mathbb{Z}}$ is equal to the set of all vertices $b$ of $\cal{F}$ with $i(a, b)=2$.
We set $b_n=h^{2n}(b_0)$ for each $n\in \mathbb{Z}$ and set $B=\{ b_n\}_{n\in \mathbb{Z}}$.
It follows that $B$ is equal to the set of all vertices $b$ in $L$ with $i(a, b)=2$.

We define $C$ as the set of all vertices $c$ of $\cal{F}$ such that $a$ and $c$ lie in a diagonal position of two adjacent triangles in $\cal{F}$.
Any element of $C$ is a vertex of $\cal{H}$ and thus of $L$, and any two elements of $C$ are sent to each other by a power of $h$.
There exists a unique element $c_0$ of $C$ with $i(b_0, c_0)=i(b_0, h(c_0))=2$.
We set $c_n=h^n(c_0)$ for each $n\in \mathbb{Z}$.
The vertices $b_n$ and $c_m$ are located as in Figure \ref{fig-hlink} (a) and satisfy desired conditions.
\end{proof}


\subsection{A simplicial graph associated to $S_{1, 2}$}\label{subsec-d}

Throughout this subsection, we put $Q=S_{1, 2}$.
The following simplicial graph associated to $Q$ is studied in \cite{kida-cohop}.

\medskip

\noindent {\bf Graph $\cal{D}$.} Let $Q=S_{1, 2}$ be a surface.
We define $\cal{D}=\cal{D}(Q)$ as the simplicial graph such that the set of vertices of $\cal{D}$ is $V_s(Q)$, and two vertices of $a$, $b$ of $\cal{D}$ are connected by an edge of $\cal{D}$ if and only if we have $i(a, b)=4$.

\medskip

We note that $\pmod(Q)$ acts transitively on the set of edges of $\cal{D}$ (see the paragraph right after the proof of Lemma 3.2 in \cite{kida-cohop}).
It follows that for any edge $\{ a, b\}$ of $\cal{D}$, there exists a unique element of $V(Q)$ disjoint from $a$ and $b$. 
We denote it by $c(a, b)$.
The curve $c(a, b)$ is non-separating in $Q$ (see Figure \ref{fig-tri}).
\begin{figure}
\begin{center}
\includegraphics[width=8cm]{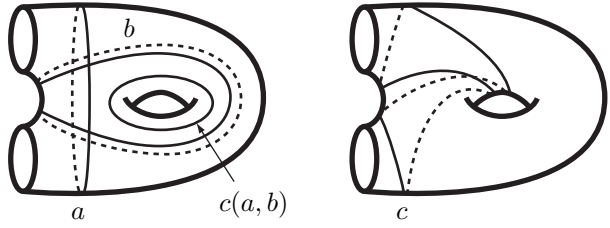}
\caption{The three curves $a$, $b$ and $c$ form a triangle in $\cal{D}$.}\label{fig-tri}
\end{center}
\end{figure}

For each vertex $a$ of $\cal{D}$, we denote by $\lk_{\cal{D}}(a)$ the link of $a$ in $\cal{D}$, and denote by $V(\lk_{\cal{D}}(a))$ the set of vertices of $\lk_{\cal{D}}(a)$.
We mean by a {\it bi-infinite line} a simplicial graph consisting of vertices $v_j$ indexed by each integer $j$ such that for any $j, k\in \mathbb{Z}$, $v_j$ and $v_k$ are adjacent if and only if $k=j-1$ or $j+1$.
Basic properties of the link of each vertex in $\cal{D}$ are summarized as follows.

\begin{prop}[\ci{Lemma 3.4}{kida-cohop}]\label{prop-dline}
In the notation in the definition of the graph $\cal{D}$, we fix a vertex $a$ of $\cal{D}$.
Let $H$ denote the handle cut off by $a$ from $Q$.
We define $h\in \mod(Q)$ as the half twist about $a$ exchanging the two components of $\partial Q$ and being the identity on $H$.
Pick two elements $u$, $v$ of $V(H)$ with $i(u, v)=1$.
We put
\[W=\{ \, w\in V(\lk_{\cal{D}}(a))\mid c(a, w)=u\, \},\quad Z=\{ \, z\in V(\lk_{\cal{D}}(a))\mid c(a, z)=v\, \}.\]
Then after an appropriate numbering, we have the equalities $W=\{ w_n\}_{n\in \mathbb{Z}}$, $Z =\{ z_m\}_{m\in \mathbb{Z}}$ and
\[h(w_n)=w_{n+1},\quad h(z_m)=z_{m+1}\quad \textrm{for\ any}\ n, m\in \mathbb{Z},\]
and the full subgraph of $\cal{D}$ spanned by all vertices of $W\cup Z$ is the bi-infinite line such that for each $n\in \mathbb{Z}$, $w_n$ is adjacent to $z_n$ and $z_{n+1}$.
\end{prop}


\subsection{A simplicial graph associated to $S_{1, 3}$}\label{subsec-g}

We define the simplicial graph $\cal{G}$ associated to $S_{1, 3}$ and show that $\cal{G}$ is a bi-infinite line.
To define the graph $\cal{G}$, we introduce sharing pairs of p-curves, motivated by sharing pairs of h-curves originally introduced by Brendle-Margalit \cite{bm} (see Definition \ref{defn-share}).

\begin{defn}\label{defn-share-p}
Let $S=S_{g, p}$ be a surface with $g\geq 1$ and $p\geq 2$.
Pick two distinct components $\partial_1$, $\partial_2$ of $\partial S$.
Let $a$ and $b$ be p-curves in $S$ cutting off $\partial_1$ and $\partial_2$.
Let $l$ and $r$ be defining arcs of $a$ and $b$, respectively, connecting $\partial_1$ and $\partial_2$.
We call the pair $\{ a, b \}$ a {\it sharing pair} in $S$ if $l$ and $r$ can be chosen so that they are disjoint and non-isotopic; and the surface obtained by cutting $S$ along $l\cup r$ is connected.
\end{defn}

If $l$ and $r$ are arcs in $S$ satisfying the conditions in Definition \ref{defn-share-p}, then $r$ is an essential simple arc in the surface obtained by cutting $S$ along $l$.
Moreover, $r$ is non-separating in that surface.
Since $a$ and $b$ are defined by $l$ and $r$, respectively, the group $\pmod(S)$ acts transitively on the set of sharing pairs in $S$ of p-curves cutting off $\partial_1$ and $\partial_2$.

Let $Q=S_{1, 2}$ be a surface.
Pick a component $\partial$ of $Q$.
We define $\cal{A}=\cal{A}(Q, \partial)$ as the set of ordered pairs $(r, r')$ of essential simple arcs in $Q$ such that
\begin{itemize}
\item each of $r$ and $r'$ connects two points of $\partial$ and is non-separating in $Q$;
\item $r$ and $r'$ are disjoint and non-isotopic; and
\item the end points of $r$ and $r'$ appear alternately along $\partial$.
\end{itemize}

\begin{rem}\label{rem-arc}
Let $(r, r')$ be an element of $\cal{A}$.
We denote by $P$ the surface obtained by cutting $Q$ along $r$, which is a pair of pants.
The component $\partial$ of $\partial Q$ is then decomposed into two arcs $s_1$, $s_2$, which are contained in distinct components of $\partial P$.
Since the end points of $r$ and $r'$ appear alternately along $\partial$, $r'$ is an essential simple arc in $P$ connecting a point of $s_1$ with a point of $s_2$.
Such an essential simple arc in $P$ uniquely exists up to a homeomorphism of $Q$ fixing $r$ as a set.
It turns out that the group of homeomorphisms of $Q$ fixing $\partial$ as a set acts on $\cal{A}$ transitively.
\end{rem}

Let $Y=S_{1, 3}$ be a surface.
Fix a p-curve $a$ in $Y$.
We denote by $Q$ the component of $Y_a$ homeomorphic to $S_{1, 2}$, and denote by $\partial$ the component of $\partial Y$ contained in $Q$.
For each essential simple arc $s$ in $Q$ which connects two points of $\partial$ and is non-separating in $Q$, we define $\Gamma(s)=\Gamma(a; s)$ as the set of all elements $\gamma$ of $V(Y)$ such that $\gamma$ is a p-curve in $Y$ cutting off the two components of $\partial Y$ distinct from $\partial$; $a$ and $\gamma$ form a sharing pair in $Y$; and a representative of $\gamma$ is disjoint from $s$.

\medskip

\noindent {\bf Graph $\cal{G}$.} Let $Y=S_{1, 3}$ be a surface.
We fix a p-curve $a$ in $Y$ and define $Q$ and $\partial$ as above.
For each element $(r, r')$ of $\cal{A}(Q, \partial)$, we define $\cal{G}=\cal{G}(Y, a; r, r')$ as the simplicial graph such that the set of vertices of $\cal{G}$ is the union $\Gamma(r)\cup \Gamma(r')$, and two vertices $\gamma$, $\gamma'$ of $\cal{G}$ are connected by an edge of $\cal{G}$ if and only if $\gamma$ and $\gamma'$ form a sharing pair in $Y$.

\begin{prop}\label{prop-gline}
In the notation in the definition of the graph $\cal{G}$, we define $h\in \mod(Y)$ as the half twist about $a$ exchanging the two components of $\partial Y$ contained in $Y_a\setminus Q$ and being the identity on $Q$.
Let $(r, r')$ be an element of $\cal{A}(Q, \partial)$.
Then after an appropriate numbering, we have the equalities $\Gamma(r)=\{ \gamma_n\}_{n\in \mathbb{Z}}$, $\Gamma(r')=\{ \gamma_m'\}_{m\in \mathbb{Z}}$ and
\[h(\gamma_n)=\gamma_{n+1},\quad h(\gamma_m')=\gamma_{m+1}'\quad \textrm{for\ any\ }n, m\in \mathbb{Z},\]
and the graph $\cal{G}$ is the bi-infinite line such that for each $n\in \mathbb{Z}$, $\gamma_n$ is adjacent to $\gamma_n'$ and $\gamma_{n+1}'$.
\end{prop}

\begin{proof}
The proof is similar to that of Proposition \ref{prop-dline}.
Put $\Gamma =\Gamma(r)$ and $\Gamma'=\Gamma(r')$.
Let $R$ denote the surface obtained by cutting $Y$ along $r$.
The set $\Gamma$ is then a subset of $V(R)$.
Let $\cal{F}(R)$ denote the graph introduced in Section \ref{subsec-farey}.
By Proposition \ref{prop-diag}, for each element $\gamma$ of $\Gamma$, $a$ and $\gamma$ lie in a diagonal position of two adjacent triangles in $\cal{F}(R)$.
Since the cyclic group generated by $h$ acts transitively on the set of triangles in $\cal{F}(R)$ containing $a$, it also acts transitively on $\Gamma$.
We thus have the equality $\Gamma =\{ h^n(\gamma_0)\}_{n\in \mathbb{Z}}$ for some $\gamma_0\in \Gamma$.
We put $\gamma_n=h^n(\gamma_0)$ for each $n\in \mathbb{Z}$.
By the criterion on geometric intersection numbers in Expos\'e 3, Proposition 10 in \cite{flp}, we have $i(\gamma_0, \gamma_n)=8|n|$ for any $n\in \mathbb{Z}$.
Since we have $i(u, v)=4$ for any two adjacent vertices $u$, $v$ of $\cal{G}$, any two distinct elements of $\Gamma$ are not adjacent in $\cal{G}$.

By Remark \ref{rem-arc}, $r$ and $r'$ are described as in Figure \ref{fig-gline}.
\begin{figure}
\begin{center}
\includegraphics[width=12cm]{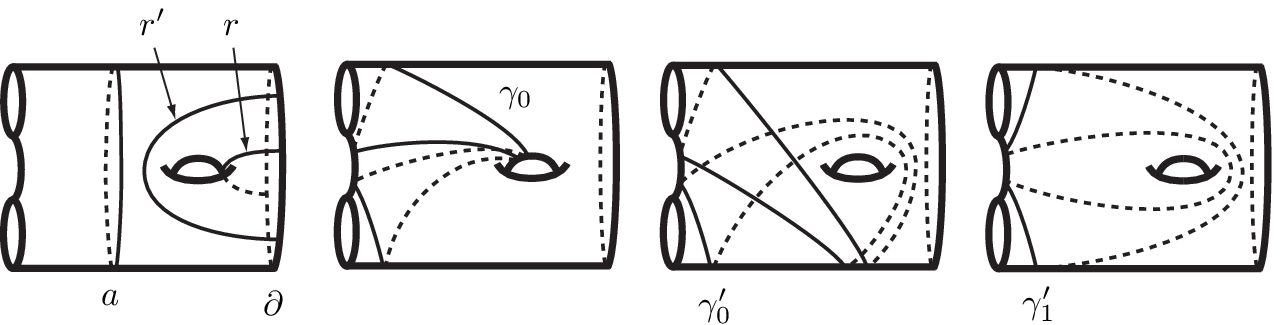}
\caption{}\label{fig-gline}
\end{center}
\end{figure}
We can find two elements $\gamma_0'$, $\gamma_1'=h(\gamma_0')$ of $\Gamma'$ adjacent to $\gamma_0$ in $\cal{G}$ as in the same Figure.
Applying the argument in the last paragraph to $\Gamma'$ and $r'$ in place of $\Gamma$ and $r$, respectively, we obtain the equality $\Gamma'=\{ \gamma_n'\}_{n\in \mathbb{Z}}$, where we put $\gamma_n'=h^n(\gamma_0')$ for each $n\in \mathbb{Z}$.
Moreover, any two distinct elements of $\Gamma'$ are not adjacent in $\cal{G}$.
By the criterion in \cite{flp} applied in the last paragraph, we have $i(\gamma_0, \gamma_n')=4|2n-1|$ for any $n\in \mathbb{Z}$.
Since for each $n\in \mathbb{Z}$, $h^n(\gamma_0)$ is adjacent to $h^n(\gamma_0')$ and $h^n(\gamma_1')$ in $\cal{G}$, the latter assertion in the lemma follows.
\end{proof}


\subsection{A simplicial graph associated to $S_{2, 1}$}\label{subsec-b}

The simplicial graph $\cal{B}$ introduced in this subsection is analogous to the graphs $\cal{D}$ and $\cal{G}$.
For an h-curve $A$ in a surface $S$ with $|\chi(S)|\geq 3$, we denote by $H_A$ the handle cut off by $A$ from $S$, which is naturally identified with a subsurface of $S$.

\begin{defn}\label{defn-share}
Let $S=S_{g, p}$ be a surface with $g\geq 2$ and $|\chi(S)|\geq 3$.
Let $\alpha, \beta \in V_s(S)$ be h-curves in $S$ and $c\in V(S)$ a non-separating curve in $S$. 
We say that $\alpha$ and $\beta$ {\it share} $c$ if there exist representatives $A$, $B$ and $C$ of $\alpha$, $\beta$ and $c$, respectively, such that we have $|A\cap B|=i(\alpha, \beta)$, $H_A\cap H_B$ is an annulus with its core curve $C$, and $S\setminus (H_A\cup H_B)$ is connected. 
In this case, $\{ \alpha, \beta \}$ is called a {\it sharing pair} for $c$ or a {\it sharing pair} in $S$ if $c$ is not specified.
\end{defn}

Sharing pairs are introduced by Brendle-Margalit \cite{bm}.
Let $c$ be a non-separating curve in $S$.
Let $\alpha$ and $\beta$ be h-curves in $S$ disjoint from $c$ and cutting off a handle containing $c$ from $S$.
The two curves $\alpha$ and $\beta$ share $c$ if and only if $\alpha$ and $\beta$ form a sharing pair in $S_c$ in the sense of Definition \ref{defn-share-p}.
It follows that $\pmod(S)$ acts transitively on the set of sharing pairs in $S$.

We put $X=S_{2, 1}$.
Let $a$ be an h-curve in $S$.
We denote by $Q$ the component of $X_a$ homeomorphic to $S_{1, 2}$, and put $\partial =\partial X$.
For each essential simple arc $s$ in $Q$ which connects two points of $\partial$ and is non-separating in $Q$, we define $B(s)=B(a; s)$ as the set of all elements $\beta$ of $V(X)$ such that $\beta$ is an h-curve in $X$; $a$ and $\beta$ form a sharing pair in $X$; and a representative of $\beta$ is disjoint from $s$.

\medskip

\noindent {\bf Graph $\cal{B}$.}
Let $X=S_{2, 1}$ be a surface.
We fix an h-curve $a$ in $X$ and define $Q$ and $\partial$ as above.
For each element $(r, r')$ of $\cal{A}(Q, \partial)$, we define $\cal{B}=\cal{B}(X, a; r, r')$ as the simplicial graph such that the set of vertices of $\cal{B}$ is the union $B(r)\cup B(r')$, and two vertices $\beta$, $\beta'$ of $\cal
{B}$ are connected by an edge of $\cal{B}$ if and only if $\beta$ and $\beta'$ form a sharing pair in $X$.

\begin{lem}\label{lem-b}
In the notation of the definition of the graph $\cal{B}$, there exists a proper subset $V$ of the union $B(r)\cup B(r')$ such that the full subgraph of $\cal{B}$ spanned by $V$ is a bi-infinite line.
\end{lem}

\begin{proof}
Let $c$ be a curve in the handle cut off by $a$ from $X$.
We denote by $Y$ the surface obtained by cutting $X$ along $c$.
We define $V$ as the set of vertices of the graph $\cal{G}=\cal{G}(Y, a; r, r')$.
The set $V$ is identified with a subset of $B(r)\cup B(r')$, and $\cal{G}$ is then the full subgraph of $\cal{B}$ spanned by $V$.
Proposition \ref{prop-gline} shows that $\cal{G}$ is a bi-infinite line.
The set $V$ is proper in $B(r)\cup B(r')$ because there exists an element $b$ of $B(r)$ such that $\{ a, b\}$ is a sharing pair for a curve in $X$ distinct from $c$.
\end{proof}


\subsection{Non-existence of injections}\label{subsec-inj}

We show results on non-existence of injective simplicial maps between graphs introduced in prior subsections.
We note that the isomorphism classes of the graphs $\cal{E}(R; \partial_1, \partial_2)$, $\cal{H}(R; \partial_1, \partial_2)$, $\cal{D}(Q)$, $\cal{G}(Y, a; r, r')$ and $\cal{B}(X, a; r, r')$ do not depend on the choice of objects inside the brackets.
We thus omit those symbols in the argument of this subsection.

\begin{lem}\label{lem-d-e}
There exists no injective simplicial map from $\cal{D}$ into $\cal{E}$. 
\end{lem}

\begin{proof}
As described in Figure \ref{fig-tri}, there exists a triangle in $\cal{D}$.
Since $\cal{E}$ is a tree by Proposition \ref{prop-e-tree}, the lemma follows.
\end{proof}

It is a plain fact that any injective simplicial map from a bi-infinite line into a bi-infinite line is surjective.
This property is crucial in the proof of the following two lemmas.

\begin{lem}\label{lem-d-h}
There exists no injective simplicial map from $\cal{D}$ into $\cal{H}$.
\end{lem}

\begin{proof}
Let $a$ be a vertex of $\cal{D}$.
There exist pairwise distinct, infinitely many bi-infinite lines in the link of $a$ in $\cal{D}$.
For in Proposition \ref{prop-dline}, if the two elements $u$, $v$ are varied, then we obtain distinct bi-infinite lines in the link of $a$ in $\cal{D}$.
On the other hand, Figure \ref{fig-hlink} (b) shows that for each vertex $a'$ of $\cal{H}$, there exists exactly one bi-infinite line in the link of $a'$ in $\cal{H}$. 
We therefore obtain the lemma.
\end{proof}

Proposition \ref{prop-gline} and Lemma \ref{lem-b} imply the following:

\begin{lem}\label{lem-b-g}
There exists no injective simplicial map from $\cal{B}$ into $\cal{G}$.
\end{lem}


\section{Preservation of the Euler characteristic}\label{sec-chi}

Let $S=S_{g, p}$ be a surface.
A simplex $\sigma$ of $\calt(S)$ is said to be {\it weakly rooted} if for each BP $b$ in $\sigma$, there exists a curve $\beta \in b$ contained in any BP in $\sigma$ that is BP-equivalent to $b$. 
A simplex $\sigma$ of $\calt(S)$ is said to be {\it rooted} if $\sigma$ consists of BPs and there is a non-separating curve $\beta$ in $S$ contained in any BP of $\sigma$. 
In this case, if $|\sigma|\geq 2$, then $\beta$ is uniquely determined and is called the {\it root curve} for $\sigma$.

We note that the maximal number of vertices in a simplex of $\calc_s(S)$ is equal to $|\chi(S)|-1=2g+p-3$ by Lemma 3.3 in \cite{kida}.

\begin{lem}\label{lem-w-rooted}
Let $S=S_{g, p}$ be a surface with $|\chi(S)|\geq 3$.
Then the following assertions hold:
\begin{enumerate}
\item Any weakly rooted simplex of $\calt(S)$ consists of at most $2g+p-3$ vertices.
\item Let $\sigma$ be a simplex of $\calt(S)$ with $|\sigma|=2g+p-3$. 
Suppose that for each $a\in \sigma$, we have a vertex $b$ of $\calt(S)$ with $i(a, b)\neq 0$ and $i(c, b)=0$ for any $c\in \sigma \setminus \{ a\}$. 
Then $\sigma$ is weakly rooted, and each component of $S_{\sigma}$ is either a handle or a pair of pants.
\end{enumerate}
\end{lem}

\begin{proof}
Let $\sigma$ be a simplex of $\calt(S)$.
We put
\[\sigma =\{ \alpha_1,\ldots, \alpha_s, b_{11},\ldots, b_{1u_1}, b_{21},\ldots, b_{t1},\ldots, b_{tu_t}\},\]
where each $\alpha_j$ is a separating curve in $S$ and each $b_{kl}$ is a BP in $S$ such that for each $k=1,\ldots, t$, the family $\{ b_{k1},\ldots, b_{ku_k}\}$ is a BP-equivalence class in $\sigma$.

We first suppose that $\sigma$ is weakly rooted.
For each $k$, there then exists a curve $\beta_{k0}$ in $S$ with $\beta_{k0}\in b_{kl}$ for any $l$.
Let $Q$ be the surface obtained by cutting $S$ along all the curves $\beta_{10}, \beta_{20},\ldots, \beta_{t0}$. 
For each $k=1,\ldots, t$ and each $l=1,\ldots, u_k$, let $\beta_{kl}$ denote the curve in $b_{kl}$ distinct from $\beta_{k0}$.
The family
\[\tau=\{ \alpha_1,\ldots, \alpha_s, \beta_{11},\ldots, \beta_{1u_1}, \beta_{21},\ldots, \beta_{t1},\ldots, \beta_{tu_t}\}\]
is a simplex of $\calc_s(Q)$. 
Since any simplex of $\calc_s(Q)$ consists of at most $2g+p-3$ vertices, we have $|\sigma|=|\tau|\leq 2g+p-3$.
Assertion (i) is proved.

We next suppose that $\sigma$ is a simplex of $\calt(S)$ satisfying $|\sigma|=2g+p-3$ and the assumption in assertion (ii).
For any $k$, $l$, we then have a curve $\gamma_{kl}$ in $b_{kl}$, but not in $b_{kl'}$ for any $l'\in \{ 1,\ldots, u_k\}\setminus \{ l\}$. 
For each $k$, we denote by $\gamma_{k0}$ the curve in $b_{k1}$ distinct from $\gamma_{k1}$. 
Let $R$ be the surface obtained by cutting $S$ along all the curves $\gamma_{10}, \gamma_{20},\ldots, \gamma_{t0}$. 
The family
\[\rho=\{ \alpha_1,\ldots, \alpha_s, \gamma_{11},\ldots, \gamma_{1u_1}, \gamma_{21},\ldots, \gamma_{t1},\ldots, \gamma_{tu_t}\}\]
is a simplex of $\calc_s(R)$. 
The assumption $|\sigma|=2g+p-3$ implies that $\rho$ is a simplex of $\calc_s(R)$ of maximal dimension.
It follows that for any $k$, $l$, the BP $b_{kl}$ consists of two of $\gamma_{k0}, \gamma_{k1},\ldots, \gamma_{ku_k}$. 
Since $b_{kl}$ does not contain $\gamma_{kl'}$ for any $l'\in \{ 1,\ldots u_k\}\setminus \{ l\}$, we have $b_{kl}=\{ \gamma_{k0}, \gamma_{kl}\}$. 
It turns out that $\sigma$ is weakly rooted. 
Since each component of $R_{\rho}$ is either a handle or a pair of pants by Lemma 3.3 in \cite{kida}, so is each component of $S_{\sigma}$.
Assertion (ii) is proved.
\end{proof}

\begin{lem}\label{lem-si-bp}
Let $S=S_{g, p}$ be a surface with $|\chi(S)|\geq 3$, and let $\phi \colon \calc_s(S)\rightarrow \calt(S)$ be a superinjective map.
Then $\phi$ sends each simplex of $\calc_s(S)$ to a weakly rooted simplex of $\calt(S)$.

Moreover, the following assertion holds:
Let $\sigma$ be a simplex of $\calc_s(S)$ and $\alpha$ an element of $\sigma$.
Let $R_1$ and $R_2$ denote the two components of $S_{\phi(\alpha)}$.
Then for each $\beta \in \sigma$, it is impossible that $\phi(\beta)$ is a BP in $S$ with one curve of $\phi(\beta)$ essential in $R_1$ and another essential in $R_2$.
\end{lem}

\begin{proof}
Let $\sigma$ be a simplex of $\calc_s(S)$.
We choose a simplex $\tau$ of $\calc_s(S)$ of maximal dimension containing $\sigma$.
Since $\phi(\tau)$ satisfies the assumption in Lemma \ref{lem-w-rooted} (ii), $\phi(\tau)$ is weakly rooted.
It follows that $\phi(\sigma)$ is also weakly rooted.
The latter assertion of the lemma follows from the former assertion and Lemma \ref{lem-bp}.
\end{proof}

\begin{lem}\label{lem-hp-number}
Let $S=S_{g, p}$ be a surface with $|\chi(S)|\geq 3$, and let $\sigma$ be a weakly rooted simplex of $\calt(S)$ with $|\sigma|=2g+p-3$. 
Then the number of vertices in $\sigma$ corresponding to either an hp-curve in $S$ or a p-BP in $S$ is at most $g+\lfloor p/2\rfloor$.
\end{lem}

\begin{proof}
We prove the lemma by induction on the number of BP-equivalence classes in $\sigma$. 
When there is no BP in $\sigma$, the lemma is obvious. 
Otherwise, we pick a BP-equivalence class $\{ b_1,\ldots, b_n\}$ in $\sigma$ and write $b_j=\{ \beta_0, \beta_j\}$ for each $j$ with $\beta_0$ the root curve.
We put $R=S_{\beta_0}$.
Let $\tau$ be the simplex of $\calt(R)$ induced by $\sigma$, that is, the simplex consisting of $\beta_1,\ldots, \beta_n$ and all elements in $\sigma \cap V_s(S)$ and in $(\sigma \setminus \{ b_1,\ldots, b_n\})\cap V_{bp}(S)$. 
The hypothesis of the induction shows the number of vertices in $\tau$ corresponding to either an hp-curve in $R$ or a p-BP in $R$ is at most $(g-1)+\lfloor (p+2)/2\rfloor =g+\lfloor p/2\rfloor$. 
Let $\alpha$ be an h-curve in $\sigma$.
Since $\alpha$ is disjoint from the BP $\{ \beta_0, \beta_1\}$, the handle cut off by $\alpha$ from $S$ does not contain $\beta_0$.
It follows that $\alpha$ is an h-curve in $R$.
Each p-curve in $\sigma$ is a p-curve in $R$.
Let $b$ be a p-BP in $\sigma$.
If $b$ is equal to $b_j$ for some $j=1,\ldots, n$, then $\beta_j$ is a p-curve in $R$.
Otherwise, $b$ is a p-BP in $R$.
The lemma thus follows.
\end{proof}

To prove the next lemma, let us recall reduction system graphs introduced in \cite{v}.
The {\it reduction system graph} $G(\tau)$ for a simplex $\tau$ of $\calc(S)$ is defined as follows.
Vertices of $G(\tau)$ are components of $S_{\tau}$.
Edges of $G(\tau)$ are curves in $\tau$.
The two ends of the edge corresponding to a curve $c$ in $\tau$ are defined to be vertices corresponding to components of $S_{\tau}$ which lie in the left and right hand sides of $c$ in $S$.
Note that $G(\tau)$ may have a loop.
We refer to \cite{v} for basics of reduction system graphs.

\begin{lem}\label{lem-pants-sep}
Let $S=S_{g, p}$ be a surface with $|\chi(S)|\geq 3$, and let $\sigma$ be a weakly rooted simplex of $\calt(S)$ with $|\sigma|=2g+p-3$. 
Choose a component $P$ of $S_{\sigma}$ such that $P$ is a pair of pants and each component of $\partial P$ corresponds to a curve of $\sigma$. 
Then at least one of the three curves corresponding to components of $\partial P$ is separating in $S$.
\end{lem}

\begin{proof}
Let $\delta_1$, $\delta_2$ and $\delta_3$ denote the elements of $\sigma$ corresponding to components of $\partial P$. 
We claim that these three elements are mutually distinct as elements of $V(S)$.
For otherwise two of them, say $\delta_2$ and $\delta_3$, would be equal and be non-separating in $S$.
The curve $\delta_1$ then cuts off a handle from $S$, in which $\delta_2=\delta_3$ is contained.
We obtain a contradiction because any non-separating curve in $S$ forming a BP in $S$ with $\delta_2$ intersects $\delta_1$.

Assuming that any of $\delta_1$, $\delta_2$ and $\delta_3$ is a non-separating curve in $S$, we deduce a contradiction.
If two of $\delta_1$, $\delta_2$ and $\delta_3$ were BP-equivalent, then another curve would be separating in $S$. 
It follows that any two of $\delta_1$, $\delta_2$ and $\delta_3$ are not BP-equivalent. 
We put
\[\sigma =\{ \alpha_1,\ldots, \alpha_s, b_{11},\ldots, b_{1u_1}, b_{21},\ldots, b_{t1},\ldots, b_{tu_t}\},\]
where each $\alpha_j$ is a separating curve in $S$ and each $b_{kl}$ is a BP in $S$ such that for each $k=1,\ldots, t$, the family $\{ b_{k1},\ldots, b_{ku_k}\}$ is a BP-equivalence class in $\sigma$. 
Let us define the simplex $\tau$ of $\calc(S)$ by
\[\tau =\{ \alpha_1,\ldots, \alpha_s\} \cup \bigcup_{k=1}^{t}\{ \beta_{k0}, \beta_{k1}\ldots, \beta_{ku_k}\},\]
where we put $b_{kl}=\{ \beta_{k0}, \beta_{kl}\}$ with $\beta_{k0}$ the root curve. 
Without loss of generality, we may assume that for each $k=1, 2, 3$, $\delta_k$ is equal to $\beta_{kl}$ for some $l$. 
Since the union of $\delta_1$, $\delta_2$ and $\delta_3$ cuts off $P$ from $S$, for any $l_1$, $l_2$ and $l_3$, the union of $\beta_{1l_1}$, $\beta_{2l_2}$ and $\beta_{3l_3}$ also decomposes $S$ into at least two components. 
It follows that for any maximal tree $T$ in the reduction system graph $G(\tau)$, there exists $r\in \{ 1, 2, 3\}$ such that all edges corresponding to $\beta_{r0}, \beta_{r1},\ldots, \beta_{ru_r}$ are contained in $T$. 
Without loss of generality, we may assume $r=1$.
Since the number of edges of $T$ is equal to $|\chi(S)|-1=2g+p-3$, we have
\[2g+p-3\geq s+(u_1+1)+\sum_{k=2}^t u_k>s+\sum_{k=1}^t u_k=|\sigma|.\]
This is a contradiction. 
\end{proof}

Let $S$ be a surface and $\phi \colon \calc_s(S)\rightarrow \calt(S)$ a superinjective map.
To state explicitly the property that $\phi$ is $\chi$-preserving, let us introduce terminology.
Fix $\alpha \in V_s(S)$.
Let $\beta$ be an element of $V_s(S)$ disjoint and distinct from $\alpha$.
If $\phi(\alpha)\in V_s(S)$, then there exists a unique component $R$ of $S_{\phi(\alpha)}$ with $\phi(\beta)\in V_t(R)$ by Lemma \ref{lem-bp}.
If $\phi(\alpha)\in V_{bp}(S)$, then there exists a unique component $R$ of $S_{\phi(\alpha)}$ such that either
\begin{itemize}
\item we have $\phi(\beta)\in V_s(R)$;
\item we have $\phi(\beta)\in V_{bp}(R)$, and the two components of $\partial R$ corresponding to curves of $\phi(\alpha)$ are contained in a single component of $R_{\phi(\beta)}$; or
\item we have $\phi(\beta)\in V_{bp}(S)$, and $\phi(\beta)$ consists of an element of $\phi(\alpha)$ and an element of $V_s(R)$
\end{itemize}
by Lemmas \ref{lem-bp} and \ref{lem-si-bp}.
Let us denote this component $R$ by $R(\phi, \alpha; \beta)$.

For each $\beta \in V_s(S)$ disjoint and distinct from $\alpha$, putting $R=R(\phi, \alpha; \beta)$, we define an element $\phi_{\alpha}(\beta)$ of $V_t(R)$ as follows:
\begin{itemize}
\item If $\phi(\beta)\in V_t(R)$, then we set $\phi_{\alpha}(\beta)=\phi(\beta)$.
\item Otherwise $\phi(\alpha)$ and $\phi(\beta)$ are BPs in $S$ sharing a curve.
We define $\phi_{\alpha}(\beta)$ as another curve in $\phi(\beta)$, which is essential and separating in $R$. 
\end{itemize}

\begin{lem}\label{lem-side}
Let $S=S_{g, p}$ be a surface with $|\chi(S)|\geq 3$, and let $\phi \colon \calc_s(S)\rightarrow \calt(S)$ be a superinjective map. 
Fix $\alpha \in V_s(S)$.
Let $Q$ be a component of $S_{\alpha}$ with $|\chi(Q)|\geq 2$.
Then there exists a component $R$ of $S_{\phi(\alpha)}$ such that for any $\beta \in V_s(Q)$, the equality $R=R(\phi, \alpha; \beta)$ holds.
\end{lem}

\begin{proof}
It is enough to show that for any $\beta_1, \beta_2\in V_s(Q)$, the equality $R(\phi, \alpha; \beta_1)=R(\phi, \alpha; \beta_2)$ holds.
We put $R_j=R(\phi, \alpha; \beta_j)$ for each $j=1, 2$.
Choose $\gamma \in V_s(Q)$ intersecting both $\beta_1$ and $\beta_2$.
Superinjectivity of $\phi$ implies that $\phi(\gamma)$ intersects both $\phi(\beta_1)$ and $\phi(\beta_2)$.
We put $R=R(\phi, \alpha; \gamma)$.
If $R_1$ and $R_2$ were distinct, then we would have $k\in \{ 1, 2\}$ with $R_k\neq R$.
Since $\phi_{\alpha}(\beta_k)$ and $\phi_{\alpha}(\gamma)$ are then disjoint, $\phi(\beta_k)$ and $\phi(\gamma)$ are also disjoint.
This is a contradiction.
\end{proof}

In the notation in Lemma \ref{lem-side}, we denote the component $R$ by $\phi_{\alpha}(Q)$.
Pick a separating curve $\alpha$ in $S$.
We first assume that $\alpha$ is an hp-curve in $S$.
Let us say that $\phi$ is {\it $\chi$-preserving at $\alpha$} if $\phi(\alpha)$ is either an hp-curve in $S$ or a p-BP in $S$.
We next assume that $\alpha$ is not an hp-curve in $S$.
Let $Q_1$ and $Q_2$ denote the two components of $S_{\alpha}$, which satisfy $|\chi(Q_1)|\geq 2$ and $|\chi(Q_2)|\geq 2$.
We say that $\phi$ is {\it $\chi$-preserving at $\alpha$} if we have $\phi_{\alpha}(Q_1)\neq \phi_{\alpha}(Q_2)$ and $\chi(\phi_{\alpha}(Q_j))=\chi(Q_j)$ for each $j=1, 2$.
If $\phi$ is $\chi$-preserving at $\gamma$ for any $\gamma \in V_s(S)$, then $\phi$ is said to be {\it $\chi$-preserving}.
To prove that $\phi$ is $\chi$-preserving, we need the following:

\begin{lem}\label{lem-chi}
Let $S=S_{g, p}$ be a surface with $|\chi(S)|\geq 4$, and let $\phi \colon \calc_s(S)\rightarrow \calt(S)$ be a superinjective map.
Let $\alpha$ be a separating curve in $S$ which is not an hp-curve in $S$.
If $\phi(\alpha)$ is neither an hp-curve in $S$ nor a p-BP in $S$, then $\phi$ is $\chi$-preserving at $\alpha$.
\end{lem}

\begin{proof}
Let $Q_1$ and $Q_2$ denote the two components of $S_{\alpha}$.
For each $j=1, 2$, we put $R_j=\phi_{\alpha}(Q_j)$.
We note that for each $j=1, 2$ and each simplex $\sigma$ of $\calc_s(Q_j)$, the family $\{\, \phi_{\alpha}(\beta)\mid \beta \in \sigma \,\}$ is a weakly rooted simplex of $\calt(R_j)$ with its cardinality $|\sigma|$.
Since for a surface $X$, any weakly rooted simplex of $\calt(X)$ consists of at most $|\chi(X)|-1$ vertices, we have $|\chi(R_j)|\geq |\chi(Q_j)|$.

If $R_1\neq R_2$, then we have $|\chi(R_j)|=|\chi(Q_j)|$ for each $j=1, 2$ because
\[|\chi(S)|=|\chi(R_1)|+|\chi(R_2)|\geq |\chi(Q_1)|+|\chi(Q_2)|=|\chi(S)|.\]
We assume $R_1=R_2$ and denote it by $R$.
For each $j=1, 2$, pick a simplex $\sigma_j$ of $\calc_s(Q_j)$ of maximal dimension.
The family $\tau =\{\, \phi_{\alpha}(\beta)\mid \beta \in \sigma_1\cup \sigma_2\,\}$ is then a weakly rooted simplex of $\calt(R)$ with
\[|\tau|=|\sigma_1|+|\sigma_2|=|\chi(Q_1)|+|\chi(Q_2)|-2=|\chi(S)|-2.\]
On the other hand, since $\phi(\alpha)$ is neither an hp-curve in $S$ nor a p-BP in $S$, we have $|\chi(R)|\leq |\chi(S)|-2$.
We thus have $|\tau|\leq |\chi(R)|-1\leq |\chi(S)|-3$.
This is a contradiction.
\end{proof}

\begin{prop}\label{prop-chi-pre}
Let $S=S_{g, p}$ be a surface with $|\chi(S)|\geq 3$.
Then any superinjective map $\phi \colon \calc_s(S)\rightarrow \calt(S)$ is $\chi$-preserving.
\end{prop}

If $|\chi(S)|=3$, then any separating curve in $S$ is an hp-curve in $S$, and any BP in $S$ is a p-BP in $S$.
The proposition thus holds.
The proposition for $S$ with $|\chi(S)|\geq 4$ is obtained by combining Lemmas \ref{lem-even-chi-pre} and \ref{lem-odd-chi-pre} below.

\begin{lem}\label{lem-even-chi-pre}
In the notation of Proposition \ref{prop-chi-pre}, we assume $|\chi(S)|\geq 4$.
Let $\alpha$ be a separating curve in $S$.
If at least one component of $S_{\alpha}$ contains an even number of components of $\partial S$, then $\phi$ is $\chi$-preserving at $\alpha$.
\end{lem}

\begin{proof}
We first prove that if $\alpha$ is an hp-curve in $S$, then $\phi(\alpha)$ is either an hp-curve in $S$ or a p-BP in $S$. 
Let $\sigma$ be a simplex of $\calc_s(S)$ of maximal dimension containing $\alpha$.
By Lemma \ref{lem-side}, we have a component $R$ of $S_{\phi(\alpha)}$ with $\phi_{\alpha}(\beta)\in V_t(R)$ for any $\beta \in \sigma \setminus \{ \alpha \}$.
The family $\{\, \phi_{\alpha}(\beta)\mid \beta \in \sigma \setminus \{ \alpha \}\,\}$ is then a weakly rooted simplex of $\calt(R)$ with its cardinality $|\sigma|-1$.
We thus have $|\chi(R)|-1\geq |\sigma|-1=|\chi(S)|-2$, and $|\chi(R)|\geq |\chi(S)|-1$.
It turns out that $\phi(\alpha)$ is either an hp-curve in $S$ or a p-BP in $S$.

We next assume that $\alpha$ is not an hp-curve in $S$.
Since at least one component of $S_{\alpha}$ contains an even number of components of $\partial S$, there exists a simplex $\tau$ of $\calc_s(S)$ consisting of $\alpha$ and $g+\lfloor p/2\rfloor$ hp-curves in $S$.
As shown in the previous paragraph, for any $\beta \in \tau \setminus \{ \alpha \}$, $\phi(\beta)$ is either an hp-curve in $S$ or a p-BP in $S$.
By Lemma \ref{lem-hp-number}, $\phi(\alpha)$ is neither an hp-curve in $S$ nor a p-BP in $S$.
It follows from Lemma \ref{lem-chi} that $\phi$ is $\chi$-preserving at $\alpha$.
\end{proof}

\begin{lem}\label{lem-odd-chi-pre}
In the notation of Proposition \ref{prop-chi-pre}, we assume $|\chi(S)|\geq 4$.
Let $\alpha$ be a separating curve in $S$.
If any component of $S_{\alpha}$ contains an odd number of components of $\partial S$, then $\phi$ is $\chi$-preserving at $\alpha$.
\end{lem}

\begin{proof}
By Lemma \ref{lem-chi}, either $\phi(\alpha)$ is an hp-curve in $S$; $\phi(\alpha)$ is a p-BP in $S$; or $\phi$ is $\chi$-preserving at $\alpha$. 
We eliminate the first two cases in Claims \ref{claim-not-hp} and \ref{claim-not-p-bp} below.
Let $\sigma$ be a simplex of $\calc_s(S)$ of maximal dimension containing $\alpha$ and the two curves $\beta_1$, $\beta_2$ described in Figure \ref{fig-chipants} (a).
\begin{figure}
\begin{center}
\includegraphics[width=9cm]{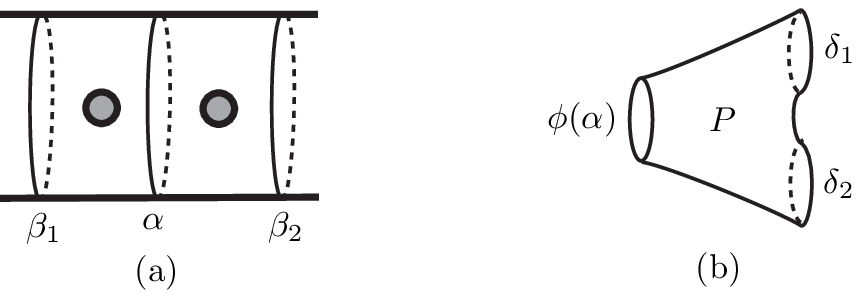}
\caption{}\label{fig-chipants}
\end{center}
\end{figure}
For each $j=1, 2$, choose $\gamma_j, \epsilon_j\in V_s(S)$ with
\[i(\gamma_j, \beta_j)\neq 0,\quad i(\gamma_j, \sigma \setminus \{ \beta_j\})=0,\]
\[i(\epsilon_j, \alpha)\neq 0,\quad i(\epsilon_j, \beta_j)\neq 0,\quad i(\epsilon_j, \sigma \setminus \{ \alpha, \beta_j\})=0.\]
In the proof of Claims \ref{claim-not-hp} and \ref{claim-not-p-bp} below, we use the following terminology and notation.
Let $v$ be a vertex of $\calt(S)$.
If $v\in V_s(S)$, then we mean by a {\it part} of $v$ the element $v$ itself.
If $v\in V_{bp}(S)$, then we mean by a {\it part} of $v$ an element of $v$.
For each $\gamma \in V(S)$, if $\gamma$ is a part of $v$, then we write $\gamma \in v$.
Otherwise, we write $\gamma \not\in v$.

\begin{claim}\label{claim-not-hp}
$\phi(\alpha)$ is not an hp-curve in $S$.
\end{claim}

\begin{proof}
Assuming that $\phi(\alpha)$ is an hp-curve in $S$, we deduce a contradiction. 
Let $P$ denote the component of $S_{\phi(\sigma)}$ containing $\phi(\alpha)$ as its boundary component and distinct from the one cut off by $\phi(\alpha)$ from $S$. 
The component $P$ is a pair of pants because otherwise $P$ would be a handle by Lemma \ref{lem-w-rooted} (ii), and the equality $|\chi(S)|=2$ would follow.

We show that no component of $\partial P$ is a component of $\partial S$.
If it were not true, then exactly one component of $\partial P$ would be a component of $\partial S$.
Let $\delta_0\in V(S)$ denote the other component of $\partial P$ distinct from $\phi(\alpha)$.
The curve $\delta_0$ is separating in $S$.
If we had $\delta_0\neq \phi(\beta_1)$, then $\phi(\epsilon_1)$ would be disjoint from $\delta_0$.
This contradicts Lemma \ref{lem-bp} because $\phi(\epsilon_1)$ intersects both $\phi(\alpha)$ and $\phi(\beta_1)$, which lie in distinct components of $S_{\delta_0}$.
We thus have $\delta_0=\phi(\beta_1)$.
We can however show the equality $\delta_0=\phi(\beta_2)$ similarly.
This is a contradiction.

Let $\delta_1, \delta_2\in V(S)$ denote the curves in $S$ corresponding to the two components of $\partial P$ other than $\phi(\alpha)$ (see Figure \ref{fig-chipants} (b)).
Since we assume $|\chi(S)|\geq 4$, we have $\delta_1\neq \delta_2$.
We put $\delta =\{ \delta_1, \delta_2\}$.
Note that $S_{\delta}$ consists of at least two components.
Either $\delta$ is a BP in $S$ or both $\delta_1$ and $\delta_2$ are separating curves in $S$.
Let $R$ denote the component of $S_{\delta}$ containing $\phi(\alpha)$.
We set $R'=S_{\delta}\setminus R$.
We aim to show that for each $j$ mod $2$, $\phi(\gamma_j)$ intersects $\delta_j$ and is disjoint from $\delta_{j+1}$ after exchanging indices if necessary. 
Once this is shown, $\phi(\gamma_1)$ and $\phi(\gamma_2)$ intersect in $P$ because they are disjoint from $\phi(\alpha)$. 
This contradicts $i(\gamma_1, \gamma_2)=0$.

We first assume that neither $\delta_1$ nor $\delta_2$ is a part of $\phi(\beta_1)$.
Any part of $\phi(\beta_1)$ is then an essential curve in a component of $R'$.
The assumption also implies that $\phi(\epsilon_1)$ is disjoint from $\delta$.
Since $\phi(\epsilon_1)$ intersects $\phi(\alpha)$ and $\phi(\beta_1)$, $\phi(\epsilon_1)$ is a BP in $S$ with one curve of $\phi(\epsilon_1)$ essential in $R$ and with another essential in a component of $R'$.
By Lemma \ref{lem-bp}, $\phi(\epsilon_1)$ and $\delta$ are BPs in $S$ and BP-equivalent.
If $\delta \in \phi(\sigma)$, then we obtain a contradiction by Lemma \ref{lem-si-bp}.
We suppose $\delta \not\in \phi(\sigma)$.
Let $\gamma_3$ and $\gamma_4$ be the elements of $\sigma$ with $\delta_1\in \phi(\gamma_3)$ and $\delta_2\in \phi(\gamma_4)$.
Each of $\gamma_3$ and $\gamma_4$ is distinct from $\alpha$ and $\beta_1$.
Since $\{ \epsilon_1, \gamma_3, \gamma_4\}$ is a simplex of $\calc_s(S)$, $\{ \phi(\epsilon_1), \phi(\gamma_3), \phi(\gamma_4)\}$ is a weakly rooted simplex of $\calt(S)$ by Lemma \ref{lem-si-bp}.
Since the BPs $\phi(\epsilon_1)$, $\phi(\gamma_3)$ and $\phi(\gamma_4)$ are mutually BP-equivalent, they share a curve in $S$.
This is a contradiction because each curve in $\phi(\epsilon_1)$ intersects either $\phi(\alpha)$ or $\phi(\beta_1)$, but each of $\phi(\gamma_3)$ and $\phi(\gamma_4)$ is disjoint from $\phi(\alpha)$ and $\phi(\beta_1)$.
We thus proved that either $\delta_1\in \phi(\beta_1)$ or $\delta_2\in \phi(\beta_1)$ holds. 
Similarly, we can show that either $\delta_1\in \phi(\beta_2)$ or $\delta_2\in \phi(\beta_2)$ holds.

Let us assume $\delta_1\in \phi(\beta_1)$. 
The following argument can be applied as well if we assume $\delta_2\in \phi(\beta_1)$. 
Assuming $\delta_1\in \phi(\beta_2)$, we first deduce a contradiction. 
The last assumption implies that $\phi(\beta_1)$ and $\phi(\beta_2)$ are BPs in $S$ and that $\phi(\epsilon_1)$ is disjoint from $\delta_1$.
Since $\phi(\sigma)$ is weakly rooted, $\delta_1$ is the root curve for the BP-equivalence class in $\phi(\sigma)$ containing $\phi(\beta_1)$ and $\phi(\beta_2)$.
We thus have $\delta \in \phi(\sigma)$.
If we had $\delta_2\not\in \phi(\beta_1)$, then $\phi(\epsilon_1)$ would be disjoint from $\delta_2$, and $\phi(\beta_1)$ would consist of $\delta_1$ and a curve in $R'$.
Since $\phi(\epsilon_1)$ intersects $\phi(\beta_1)$, no part of $\phi(\epsilon_1)$ is an essential curve in $R$ by Lemma \ref{lem-si-bp}.
This is a contradiction because $\phi(\epsilon_1)$ intersects $\phi(\alpha)$.
We thus have $\delta_2\in \phi(\beta_1)$ and $\phi(\beta_1)=\{ \delta_1, \delta_2\}$. 
However, $\phi(\epsilon_2)$ is then disjoint from $\delta_1$ and $\delta_2$ because we have $i(\epsilon_2, \beta_1)=0$.
The BP $\phi(\beta_2)$ consists of $\delta_1$ and a curve in $R'$.
Since $\phi(\epsilon_2)$ intersects $\phi(\beta_2)$, no part of $\phi(\epsilon_2)$ is an essential curve in $R$ by Lemma \ref{lem-si-bp}.
This is a contradiction because $\phi(\epsilon_2)$ intersects $\phi(\alpha)$.
We have proved $\delta_1\not\in \phi(\beta_2)$ and thus $\delta_2\in \phi(\beta_2)$.

Similarly, we can prove $\delta_2\not\in \phi(\beta_1)$ by using the condition $\delta_2\in \phi(\beta_2)$. 
We now show that for each $j$ mod $2$, $\phi(\gamma_j)$ intersects $\delta_j$ and is disjoint from $\delta_{j+1}$. 
This is true if both $\delta_1$ and $\delta_2$ are separating curves in $S$, for we then have $\delta_1=\phi(\beta_1)$ and $\delta_2=\phi(\beta_2)$. 
We assume that $\delta =\{ \delta_1, \delta_2\}$ is a BP in $S$. 
It then follows that $\phi(\beta_1)$ and $\phi(\beta_2)$ are BPs in $S$ and are BP-equivalent. 
The root curve for $\{ \phi(\beta_1), \phi(\beta_2)\}$ is distinct from $\delta_1$ and $\delta_2$ because we have $\delta_1\not\in \phi(\beta_2)$ and $\delta_2\not\in \phi(\beta_1)$. 
It turns out that for each $j$ mod $2$, $\phi(\gamma_j)$ intersects $\delta_j$ and is disjoint from $\delta_{j+1}$.

As indicated in the third paragraph of the proof, $\phi(\gamma_1)$ and $\phi(\gamma_2)$ intersects in $P$, and this contradicts $i(\gamma_1, \gamma_2)=0$.
\end{proof}

\begin{claim}\label{claim-not-p-bp}
$\phi(\alpha)$ is not a p-BP in $S$.
\end{claim}

\begin{proof}
Assuming that $\phi(\alpha)$ is a p-BP in $S$, we deduce a contradiction.

We first suppose that the set $\{ \phi(\alpha)\}$ is a BP-equivalence class in $\phi(\sigma)$.
Since any separating curve or BP in $\phi(\sigma)\setminus \{ \phi(\alpha)\}$ cuts off a surface containing the two curves of $\phi(\alpha)$ from $S$, there exist exactly two components of $S_{\phi(\sigma)}$ homeomorphic to a pair of pants and containing the two curves in $\phi(\alpha)$ as their boundary components.
Let $P$ denote the one of them distinct from the one cut off by $\phi(\alpha)$ from $S$.
Since we assume $|\chi(S)|\geq 4$, the component of $\partial P$ distinct from the two curves in $\phi(\alpha)$ is not a component of $\partial S$.
Let $\delta \in V(S)$ denote that component of $\partial P$ (see Figure \ref{fig-chipdec} (a)).
The curve $\delta$ is separating in $S$.
\begin{figure}
\begin{center}
\includegraphics[width=10cm]{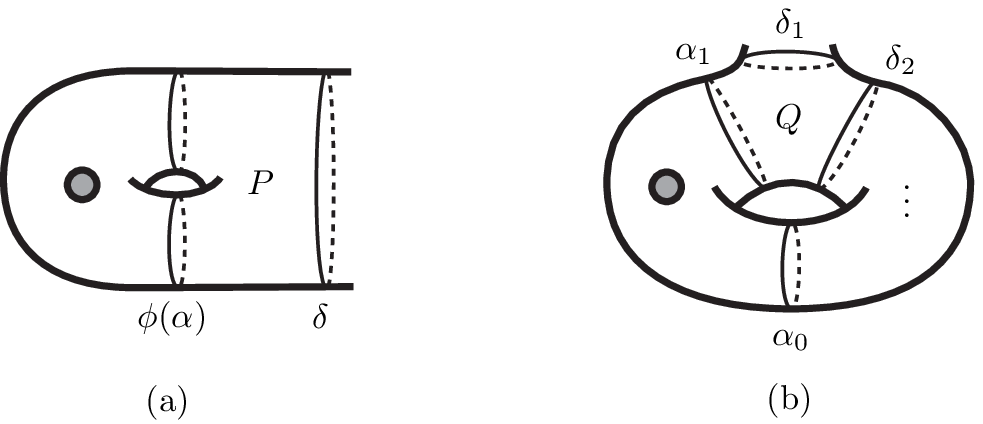}
\caption{}\label{fig-chipdec}
\end{center}
\end{figure}
If $\delta \neq \phi(\beta_1)$, then $\phi(\epsilon_1)$ would be disjoint from $\delta$.
This contradicts Lemma \ref{lem-bp} because $\phi(\epsilon_1)$ intersects both $\phi(\alpha)$ and $\phi(\beta_1)$, which lie in distinct components of $S_{\delta}$.
We thus have $\delta =\phi(\beta_1)$.
We can however show the equality $\delta =\phi(\beta_2)$ similarly.
This is a contradiction.

We next suppose that the BP-equivalence class in $\phi(\sigma)$ containing $\phi(\alpha)$ consists of at least two BPs in $S$.
Let $\alpha_0$ denote the root curve for that class.
We define a curve $\alpha_1$ in $S$ by the equality $\phi(\alpha)=\{ \alpha_0, \alpha_1\}$.
Let $Q$ denote the component of $S_{\phi(\sigma)}$ containing $\alpha_1$ as its boundary component and distinct from the pair of pants cut off by $\phi(\alpha)$ from $S$.

We show that no component of $\partial Q$ is a component of $\partial S$.
If it were not true, then exactly one component of $\partial Q$ would be a component of $\partial S$.
Let $\delta_0\in V(S)$ denote the other component of $\partial Q$ distinct from $\alpha_1$.
Since $\{ \delta_0, \alpha_1\}$ cuts off a pair of pants from $S$, the curve $\delta_0$ is non-separating in $S$ and BP-equivalent to $\alpha_1$.
Since we assume $|\chi(S)|\geq 4$, we have $\delta_0\neq \alpha_0$.
The BP $\{ \alpha_0, \delta_0\}$ thus belongs to $\phi(\sigma)$.
Let $R$ be the component of $S_{\{ \alpha_0, \delta_0\}}$ containing $\alpha_1$, and put $R'=S_{\{ \alpha_0, \delta_0\}}\setminus R$.
If we had $\{ \alpha_0, \delta_0\}\neq \phi(\beta_1)$, then $\phi(\epsilon_1)$ would be disjoint from $\{ \alpha_0, \delta_0\}$, and any part of $\phi(\beta_1)$ would be either $\alpha_0$ or an essential curve in $R'$.
Since $\phi(\epsilon_1)$ intersects $\phi(\beta_1)$, no part of $\phi(\epsilon_1)$ is an essential curve in $R$ by Lemma \ref{lem-si-bp}.
This is a contradiction because $\phi(\epsilon_1)$ intersects $\alpha_1$.
We thus have $\{ \alpha_0, \delta_0\}=\phi(\beta_1)$.
We can however show the equality $\{ \alpha_0, \delta_0\}=\phi(\beta_2)$ similarly.
This is a contradiction.

Let $\delta_1, \delta_2\in V(S)$ denote the curves in $S$ corresponding to the two components of $\partial Q$ distinct from $\alpha_1$.
The two curves $\delta_1$ and $\delta_2$ are distinct as elements of $V(S)$ because otherwise $\alpha_1$ would be an h-curve in $S$.

If both $\delta_1$ and $\delta_2$ were separating in $S$, then $\alpha_1$ would be separating in $S$.
This is a contradiction.
Without loss of generality, we may assume that $\delta_2$ is non-separating in $S$.
By Lemma \ref{lem-pants-sep}, $\delta_1$ is separating in $S$, and thus $\alpha_1$ and $\delta_2$ form a BP in $S$ (see Figure \ref{fig-chipdec} (b)).
Note that $\delta_2\neq \alpha_0$ because otherwise $\{ \phi(\alpha)\}$ would be a BP-equivalence class in $\phi(\sigma)$.
We thus have $\{ \alpha_0, \delta_2\}\in \phi(\sigma)$.

Let $T$ denote the component of $S_{\{ \alpha_0, \delta_1, \delta_2\}}$ containing $\alpha_1$ and put $T'=S_{\{ \alpha_0, \delta_1, \delta_2\}}\setminus T$.
If $\phi(\beta_1)$ were equal to neither $\delta_1$ nor $\{ \alpha_0, \delta_2\}$, then $\phi(\epsilon_1)$ would be disjoint from $\alpha_0$, $\delta_1$ and $\delta_2$, and any part of $\phi(\beta_1)$ would be either $\alpha_0$ or an essential curve in a component of $T'$.
Since $\phi(\epsilon_1)$ intersects $\phi(\beta_1)$, no part of $\phi(\epsilon_1)$ is an essential curve in $T$ by Lemmas \ref{lem-bp} and \ref{lem-si-bp}.
This is a contradiction because $\phi(\epsilon_1)$ intersects $\alpha_1$.
It follows that $\phi(\beta_1)$ is equal to either $\delta_1$ or $\{ \alpha_0, \delta_2\}$.
Similarly, we can show that $\phi(\beta_2)$ is also equal to either $\delta_1$ or $\{ \alpha_0, \delta_2\}$.

Without loss of generality, we may put $\phi(\beta_1)=\delta_1$ and $\phi(\beta_2)=\{ \alpha_0, \delta_2\}$.
We then have
\[i(\phi(\gamma_1), \delta_1)\neq 0,\quad i(\phi(\gamma_2), \delta_2)\neq 0,\]
\[i(\phi(\gamma_1), \alpha_1)=i(\phi(\gamma_1), \delta_2)=i(\phi(\gamma_2), \alpha_1)=i(\phi(\gamma_2), \delta_1)=0.\]
It follows that $\phi(\gamma_1)$ and $\phi(\gamma_2)$ intersect in $Q$.
This contradicts $i(\gamma_1, \gamma_2)=0$.
\end{proof}

Claims \ref{claim-not-hp} and \ref{claim-not-p-bp} complete the proof of Lemma \ref{lem-odd-chi-pre}.
\end{proof}


\section{$S_{g, p}$ with $g\geq 3$ and $p\leq 1$, and $S_{2, 2}$}\label{sec-p1}

Let $S=S_{g, p}$ be a surface, and let $\phi \colon \calc_s(S)\rightarrow \calt(S)$ be a superinjective map. 
When $p=0$, the inclusion $\phi(\calc_s(S))\subset \calc_s(S)$ is readily shown in Section \ref{subsec-p0}. 
When $p\geq 1$, we prove the inclusion by induction on $p$ throughout Sections \ref{sec-p1}--\ref{sec-other}. 
Sections \ref{subsec-p1}--\ref{subsec-conc} are devoted to the first step of the induction for surfaces of genus at least two. 
Namely, we deal with $S_{g, 1}$ with $g\geq 3$ and $S_{2, 2}$. 
In Section \ref{sec-14}, the inclusion is proved for $S_{1, 4}$, and this is the first step of the induction for surfaces of genus one. 
In Section \ref{sec-other}, the induction is completed.

\subsection{The case $p=0$}\label{subsec-p0}

If $S=S_{g, 0}$ is a closed surface with $g\geq 3$, then there exists neither p-curve in $S$ nor p-BP in $S$. 
This is why the proof of the inclusion $\phi(\calc_s(S))\subset \calc_s(S)$ for a closed surface is much easier than the other cases.

\begin{prop}
Let $S=S_{g, 0}$ be a closed surface with $g\geq 3$, and let $\phi \colon \calc_s(S)\rightarrow \calt(S)$ be a superinjective map. 
Then the inclusion $\phi(\calc_s(S))\subset \calc_s(S)$ holds.
\end{prop}

\begin{proof}
By Proposition \ref{prop-chi-pre}, $\phi$ preserves h-curves in $S$. 
Let $\alpha$ be a separating curve in $S$ which is not an h-curve in $S$. 
We can choose $g$ h-curves $\beta_1,\ldots, \beta_g$ in $S$ forming a $g$-simplex of $\calc_s(S)$ together with $\alpha$.
For each $j=1,\ldots, g$, $\phi(\beta_j)$ is an h-curve in $S$ disjoint from $\phi(\alpha)$.
It is thus impossible that $\phi(\alpha)$ is a BP in $S$.
\end{proof}


\subsection{The case $g\geq 3$ and $p=1$}\label{subsec-p1}

Let $S=S_{g, 1}$ be a surface with $g\geq 3$, and let $\phi \colon \calc_s(S)\rightarrow \calt(S)$ be a superinjective map. 
We fix an h-curve $a$ in $S$ and assume that $\phi(a)$ is a p-BP in $S$. 
We then deduce a contradiction throughout this subsection and Section \ref{subsec-share}. 
In Section \ref{subsec-share}, we deal with $S_{g, 1}$ with $g\geq 3$ and $S_{2, 2}$ together.

\begin{lem}\label{lem-sep-bp-eq}
If $z$ is a separating curve in $S$ such that
\begin{itemize}
\item $i(z, a)=0$ and $z\neq a$; and
\item the component of $S_z$ containing $a$ does not contain $\partial S$,
\end{itemize}
then $\phi(z)$ is a BP in $S$ which is BP-equivalent to $\phi(a)$. 
\end{lem}

\begin{proof}
Let $Q$ be the component of $S_z$ containing $a$.
Let $R$ be the component of $S_{\phi(z)}$ containing at least one curve of $\phi(a)$ as an essential one. 
We prove the lemma by induction on the genus of $Q$.
The genus of $Q$ is at least two because $Q$ contains $a$ as an essential curve.

Suppose that the genus of $Q$ is equal to two.
The surface $Q$ is homeomorphic to $S_{2, 1}$. 
By Proposition \ref{prop-chi-pre}, we have $|\chi(R)|=|\chi(Q)|=3$.
Assuming that $\phi(z)$ is either a separating curve in $S$ or a BP in $S$ which is not BP-equivalent to $\phi(a)$, we deduce a contradiction. 
If $\phi(z)$ were a separating curve in $S$, then $|\chi(R)|$ would be even because the number of components of $\partial R$ is equal to two. 
This is a contradiction. 
We next assume that $\phi(z)$ is a BP in $S$ which is not BP-equivalent to $\phi(a)$. Let $R'$ denote the component of $R_{\phi(a)}$ distinct from the pair of pants cut off by $\phi(a)$ from $S$. 
Since we have $|\chi(R')|=2$, $R'$ is homeomorphic to $S_{0, 4}$, and $\partial R'$ consists of the four curves in $\phi(a)\cup \phi(z)$. 
Choose an h-curve $b$ in $Q$ disjoint and distinct from $a$. 
By Proposition \ref{prop-chi-pre}, $\phi(b)$ is either an hp-curve in $S$ or a p-BP in $S$.
This is impossible because $\phi(b)$ is disjoint from $\phi(a)$ and $\phi(z)$.  
It follows that $\phi(z)$ is a BP in $S$ which is BP-equivalent to $\phi(a)$.

When the genus of $Q$, denoted by $g_1$, is more than two, choose a separating curve $z_1$ in $S$ such that
\begin{itemize}
\item $i(z_1, a)=i(z_1, z)=0$, $z_1\neq a$ and $z_1\neq z$; and 
\item one component of $S_{z_1}$ contains $a$ and is homeomorphic to $S_{g_1-1, 1}$, and another component contains $z$.
\end{itemize}
The hypothesis of the induction implies that $\phi(z_1)$ is a BP in $S$ which is BP-equivalent to $\phi(a)$. 
Note that $\phi(a)$ and $\phi(z_1)$ share a curve by Lemma \ref{lem-si-bp}. 
The same argument as in the previous paragraph shows that $\phi(z)$ is a BP in $S$ which is BP-equivalent to $\phi(z_1)$. 
The induction is completed.
\end{proof}

\begin{lem}\label{lem-sep-pre-top}
If $z$ is a separating curve in $S$ such that
\begin{itemize}
\item $i(z, a)=0$ and $z\neq a$; and
\item the component of $S_z$ containing $a$ contains $\partial S$,
\end{itemize}
then $\phi(z)$ is a separating curve in $S$.
Moreover, $\phi$ preserves the topological type of $z$. 
Namely, if we denote by $Q_1$ and $Q_2$ the two components of $S_z$, then for each $j=1, 2$, $Q_j$ and $\phi_z(Q_j)$ are homeomorphic, where we use the notation introduced right after the proof of Lemma \ref{lem-side}.
\end{lem}

\begin{proof}
We assume that $\phi(z)$ is a BP in $S$. 
The component of $S_z$, denoted by $Q$, that contains $a$ has exactly two boundary component.
We put $R=\phi_z(Q)$, which is the component of $S_{\phi(z)}$ containing at least one curve of $\phi(a)$ as an essential one.
By our assumption, $R$ has exactly three boundary components. 
It follows that $Q$ and $R$ cannot have the same Euler characteristic.
This contradicts Proposition \ref{prop-chi-pre}. 
We proved that $\phi(z)$ is a separating curve in $S$. 

The latter assertion in the lemma follows because $\phi$ is $\chi$-preserving and $S$ has exactly one boundary component.
\end{proof}


\subsection{The case $g=2$ and $p=2$}

This subsection is preliminary to Sections \ref{subsec-share} and \ref{subsec-conc}, where we prove the inclusion $\phi(\calc_s(S))\subset \calc_s(S)$ for $S=S_{g, 1}$ with $g\geq 3$ and $S=S_{2, 2}$ together.

\begin{lem}\label{lem-22-pre}
Let $S=S_{2, 2}$ be a surface and $\phi \colon \calc_s(S)\rightarrow \calt(S)$ a superinjective map.
Choose hp-curves $\alpha$, $\beta$ and $\gamma$ in $S$ which are mutually disjoint and distinct. 
We assume that at least one of $\phi(\alpha)$, $\phi(\beta)$ and $\phi(\gamma)$ is a BP in $S$. 
Then the following assertions hold:
\begin{figure}
\begin{center}
\includegraphics[width=11cm]{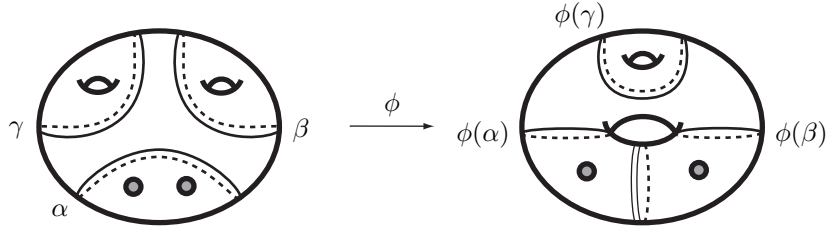}
\caption{The double line is the root curve for $\{ \phi(\alpha), \phi(\beta)\}$.}\label{fig-22}
\end{center}
\end{figure}
\begin{enumerate}
\item Exactly two of $\phi(\alpha)$, $\phi(\beta)$ and $\phi(\gamma)$ are BPs in $S$ which are BP-equivalent, and another is an h-curve in $S$.
\item If $\alpha$ is a p-curve in $S$, then $\phi(\alpha)$ is a BP in $S$ (see Figure \ref{fig-22}).
\end{enumerate}  
\end{lem}

\begin{proof}
By Proposition \ref{prop-chi-pre}, each of $\phi(\alpha)$, $\phi(\beta)$ and $\phi(\gamma)$ is either an hp-curve in $S$ or a p-BP in $S$.
To prove assertion (i), let us assume that $\phi(\alpha)$ is a p-BP in $S$.
It is impossible that both $\phi(\beta)$ and $\phi(\gamma)$ are hp-curves in $S$ and that both $\phi(\beta)$ and $\phi(\gamma)$ are p-BPs in $S$.
Exactly one of $\phi(\beta)$ and $\phi(\gamma)$ is thus a p-BP in $S$.

Let us assume that $\phi(\beta)$ is a p-BP in $S$. 
If $\phi(\alpha)$ and $\phi(\beta)$ were not BP-equivalent, then $\phi(\gamma)$ could not be an hp-curve in $S$ because any vertex of $V_t(S)$ forming a 2-simplex of $\calt(S)$ with $\phi(\alpha)$ and $\phi(\beta)$ corresponds to a separating curve in $S$ which decomposes $S$ into two surfaces homeomorphic to $S_{1, 2}$.
This is a contradiction.
The BPs $\phi(\alpha)$ and $\phi(\beta)$ are thus BP-equivalent.
It follows that $\phi(\gamma)$ is an h-curve in $S$.
Assertion (i) is proved.

To prove assertion (ii), we next assume that $\alpha$ is a p-curve in $S$. 
If $\phi(\alpha)$ were not a p-BP in $S$, then assertion (i) would imply that $\phi(\alpha)$ is an h-curve in $S$ and that both $\phi(\beta)$ and $\phi(\gamma)$ are p-BPs in $S$ which are BP-equivalent. 
Choose a separating curve $\delta$ in $S$ such that
\begin{itemize}
\item $i(\delta, \alpha)\neq 0$ and $i(\delta, \beta)=i(\delta, \gamma)=0$; and
\item both components of $S_{\delta}$ are homeomorphic to $S_{1, 2}$.
\end{itemize}
Any vertex of $V_t(S)$ forming a 2-simplex of $\calt(S)$ with $\phi(\beta)$ and $\phi(\gamma)$ corresponds to either the BP in $S$ consisting of the curves in $\phi(\beta)\cup \phi(\gamma)$ distinct from the root curve for $\{ \phi(\beta), \phi(\gamma)\}$ or an h-curve in $S$.
It follows that $\phi(\delta)$ is an h-curve in $S$ because $\phi(\delta)$ intersects $\phi(\alpha)$. 
This is a contradiction because $\phi$ is $\chi$-preserving. 
\end{proof}


\subsection{Construction of injective simplicial maps}\label{subsec-share}

Let $S=S_{g, p}$ be a surface with either $g\geq 3$ and $p=1$ or $g=2$ and $p=2$. 
Let $\phi \colon \calc_s(S)\rightarrow \calt(S)$ be a superinjective map. 
Fix an h-curve $a$ in $S$. 
The aim of this subsection is to deduce a contradiction on the assumption that $\phi(a)$ is a p-BP in $S$.

Let $b$, $x$, $y$, $z$ and $w$ denote the curves in $S$ described in Figure \ref{fig-s-pair-cha}.
\begin{figure}
\begin{center}
\includegraphics[width=12cm]{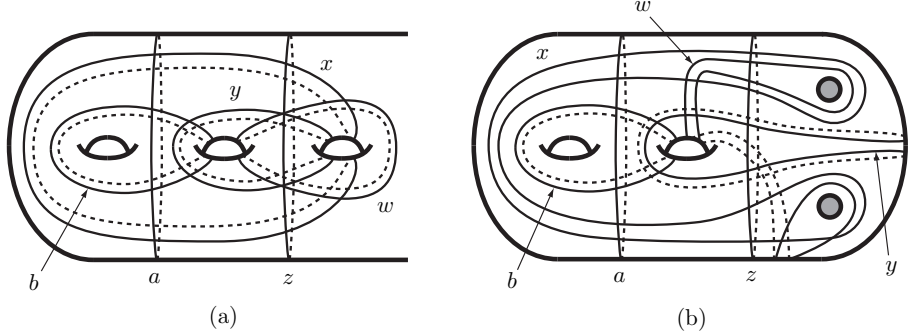}
\caption{(a) $S=S_{g, 1}$ with $g\geq 3$; (b) $S=S_{2, 2}$}\label{fig-s-pair-cha}
\end{center}
\end{figure}
The choices of curves in (a) and (b) of Figure \ref{fig-s-pair-cha} are introduced in Figure 2 of \cite{bm-add} and Figure 10 of \cite{kida}, respectively, to characterize sharing pairs in terms of disjointness and non-disjointness. 
We note that
\begin{itemize}
\item $a$ and $b$ are h-curves in $S$ forming a sharing pair in $S$, and thus $i(a, b)\neq 0$;
\item $z$ cuts off a surface homeomorphic to $S_{2, 1}$ and containing $a$ and $b$;
\item $x$ cuts off a surface homeomorphic to $S_{2, 1}$ and containing $b$;
\item $y$ is an h-curve in $S$;
\item $w$ is an h-curve in $S$ if $g\geq 3$, and is a p-curve in $S$ if $S=S_{2, 2}$;
\item $i(w, a)=0$, $i(w, b)=0$ and $i(w, z)\neq 0$;
\item $i(x, a)\neq 0$, $i(x, b)=0$ and $i(x, z)\neq 0$;
\item $i(y, a)=0$, $i(y, b)\neq 0$ and $i(y, z)\neq 0$; and
\item $i(x, y)=0$.
\end{itemize}
Lemmas \ref{lem-sep-bp-eq} and \ref{lem-sep-pre-top} imply that if $g\geq 3$, then
\begin{itemize}
\item $\phi(z)$ is a BP in $S$ which is BP-equivalent to $\phi(a)$; and
\item both $\phi(y)$ and $\phi(w)$ are h-curves in $S$.
\end{itemize}
Lemma \ref{lem-22-pre} implies that if $S=S_{2, 2}$, then
\begin{itemize}
\item $\phi(z)$ is a BP in $S$ which is BP-equivalent to $\phi(a)$;
\item $\phi(y)$ is an h-curve in $S$; and
\item $\phi(w)$ is a BP in $S$ which is BP-equivalent to $\phi(a)$, and the three BPs $\phi(a)$, $\phi(z)$ and $\phi(w)$ share a non-separating curve in $S$.
\end{itemize}
The last property is verified as follows: If the three BPs $\phi(a)$, $\phi(z)$ and $\phi(w)$ did not share a non-separating curve in $S$, then the inclusion $\phi(a)\subset \phi(z)\cup \phi(w)$ would hold because $\{ \phi(a), \phi(z)\}$ and $\{ \phi(a), \phi(w)\}$ are rooted. 
This is a contradiction because $\phi(b)$ intersects $\phi(a)$, but is disjoint from $\phi(z)$ and $\phi(w)$.

We introduce several symbols employed in the rest of this subsection. 
Figure \ref{fig-sc} will be helpful.
\begin{figure}
\begin{center}
\includegraphics[width=12cm]{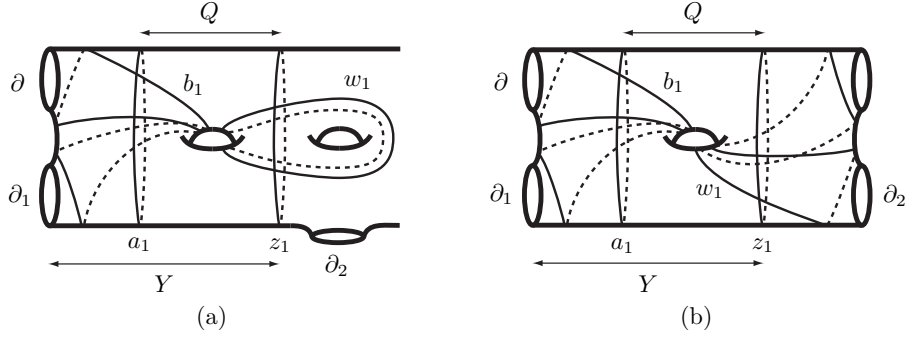}
\caption{(a) $S=S_{g, 1}$ with $g\geq 3$; (b) $S=S_{2, 2}$}\label{fig-sc}
\end{center}
\end{figure}
Let $c$ be the root curve for $\{ \phi(a), \phi(z)\}$. 
We define two curves $a_1$ and $z_1$ in $S$ by the equalities $\phi(a)=\{ c, a_1\}$ and $\phi(z)=\{ c, z_1\}$.
We put $y_1=\phi(y)$.
We put $w_1=\phi(w)$ if $g\geq 3$, and define a curve $w_1$ in $S$ by the equality $\phi(w)=\{ c, w_1\}$ if $S=S_{2, 2}$. 
Let $\partial$ denote the component of $\partial S$ contained in the pair of pants cut off by $\phi(a)$ from $S$. 
Let $\partial_1$ and $\partial_2$ denote the two boundary components of $S_c$ corresponding to $c$, where $\partial_1$ is chosen so that $a_1$ cuts off a pair of pants containing $\partial$ and $\partial_1$ from $S_c$. 
Let $Q$ denote the component of $S_{\phi(a)\cup \phi(z)}$ that contains neither $\partial_1$ nor $\partial_2$.
Let $Y$ denote the component of $S_{\phi(z)}$ containing $\partial$.

Fixing representatives of the isotopy classes of curves chosen above so that any two of them intersect minimally, we denote them by the same symbols as their isotopy classes.
This ambiguity will be of no importance in the sequel.

\begin{lem}\label{lem-phib-not-h}
$\phi(b)$ is not an h-curve in $S$.
\end{lem}

\begin{proof}
Assume that $\phi(b)$ is an h-curve in $S$. 
The curve $\phi(b)$ is an h-curve in $Y$ because $\phi(b)$ is disjoint from $\phi(z)$. 
Note that $w_1$ intersects $z_1$ and is disjoint from $a_1$ and $\phi(b)$.
The intersection $w_1\cap Y$ consists of essential simple arcs in $Y$ because $w_1$ and $z_1$ intersect minimally.
Since $\phi(b)$ is an h-curve in $Y$, any component of $w_1\cap Y$ is separating in $Y$.
Since we have $w_1\cap Y=w_1\cap Q$, any component of $w_1\cap Q$ is separating in $Q$.
This is impossible because $\phi(b)$ is an h-curve in $Y$ disjoint from $w_1$ and $z_1$, but intersecting $a_1$.
\end{proof}

\begin{lem}\label{lem-phi-b-c}
$\phi(b)$ is a p-BP in $S$ containing $c$.
\end{lem}

\begin{proof}
When $g\geq 3$, Lemma \ref{lem-phib-not-h} implies that $\phi(b)$ is a p-BP in $S$ because there is no p-curve in $S$. 
By Lemma \ref{lem-sep-bp-eq}, $\phi(b)$ and $\phi(z)$ are BP-equivalent. 
When $S=S_{2, 2}$, Lemmas \ref{lem-22-pre} and \ref{lem-phib-not-h} imply that $\phi(b)$ is a p-BP in $S$ which is BP-equivalent to $\phi(z)$. 
It follows that in both cases, $\{ \phi(b), \phi(z)\}$ is rooted. 
If the three BPs $\phi(a)$, $\phi(b)$ and $\phi(z)$ did not share a non-separating curve in $S$, then we would have the inclusion $\phi(z)\subset \phi(a)\cup \phi(b)$. 
This is a contradiction because $\phi(w)$ is disjoint from $\phi(a)$ and $\phi(b)$, but intersects $\phi(z)$.
\end{proof}

We define a curve $b_1$ in $S$ by the equality $\phi(b)=\{ c, b_1\}$.

\begin{lem}
$\phi(x)$ is a BP in $S$ which is BP-equivalent to $\phi(b)$ and contains $c$.
\end{lem}

\begin{proof}
Since $\phi(b)$ is a p-BP in $S$ disjoint from $\phi(x)$, Lemmas \ref{lem-sep-bp-eq} and \ref{lem-22-pre} imply that $\phi(x)$ is a BP in $S$ which is BP-equivalent to $\phi(b)$.
By Lemma \ref{lem-si-bp}, $\{ \phi(b), \phi(x)\}$ is rooted.
If $c$ were not contained in $\phi(x)$, then $b_1$ would be contained in $\phi(x)$. 
The inclusion $\phi(b)\subset \phi(a)\cup \phi(x)$ thus holds.
This is a contradiction because $\phi(y)$ is disjoint from $\phi(a)$ and $\phi(x)$, but intersects $\phi(b)$.
\end{proof}

We define a curve $x_1$ in $S$ by the equality $\phi(x)=\{ c, x_1\}$.

\begin{lem}\label{lem-i4}
The two p-curves $a_1$, $b_1$ in $Y$ form a sharing pair in $Y$.
Namely, there exist essential simple arcs $l_a$, $l_b$ in $Y$ satisfying the following conditions:
\begin{itemize}
\item each of $l_a$ and $l_b$ connects $\partial$ and $\partial_1$;
\item $l_a$ is disjoint from $a_1$; $l_b$ is disjoint from $b_1$;
\item $l_a$ and $l_b$ are disjoint and non-isotopic; and
\item the surface obtained by cutting $Y$ along $l_a\cup l_b$ is connected.
\end{itemize}
\end{lem}

\begin{proof}
The following argument is inspired by the proof of Lemma 4.1 of \cite{bm}.

\begin{claim}\label{claim-lw}
The intersection $w_1\cap Q$ consists of mutually isotopic, essential simple arcs in $Q$ which are non-separating in $Q$.
\end{claim}

\begin{proof}
Each component of $b_1\cap Q$ is an essential simple arc in $Q$ which is non-separating in $Q$.
For if there were a component of $b_1\cap Q$ which is separating in $Q$, then there would be no room for $w_1\cap Q$ because $w_1$ is disjoint from $a_1$ and $b_1$.
Cutting $Q$ along a component of $b_1\cap Q$, we obtain a pair of pants.
The claim follows because each component of $w_1\cap Q$ is an essential simple arc in this pair of pants.
\end{proof}

Let $l_w$ be a component of $w_1\cap Q$.

\begin{claim}\label{claim-ycapr}
Each component of $y_1\cap Y$ is an essential simple arc in $Y$ which is non-separating in $Y$.
At least one component of $y_1\cap Y$ is not isotopic to $l_w$.
\end{claim}

\begin{proof}
We first assume that there were a component $l$ of $y_1\cap Y$ which is separating in $Y$.
The arc $l$ is disjoint from $a_1$ and is thus an essential simple arc in $Q$ which is separating in $Q$.
Since $x_1$ is disjoint from $y_1$ and intersects $a_1$, there exists a component of $x_1\cap Y$ cutting off an annulus from $Y$.
The annulus contains exactly one of $\partial$ and $\partial_1$.
This contradicts the existence of the p-curve $b_1$ in $Y$ cutting off $\partial$ and $\partial_1$ and disjoint from $x_1$. 
It follows that each component of $y_1\cap Y$ is non-separating in $Y$.

If any component of $y_1\cap Y$ were isotopic to $l_w$, then $b_1$ would be disjoint from $y_1$ because $b_1$ is disjoint from $w_1$. 
This contradicts $i(\phi(b), \phi(y))\neq 0$. 
\end{proof}

The following claim is verified by a verbatim argument after exchanging $a_1$ and $y_1$ with $b_1$ and $x_1$, respectively.

\begin{claim}
Each component of $x_1\cap Y$ is an essential simple arc in $Y$ which is non-separating in $Y$.
At least one component of $x_1\cap Y$ is not isotopic to $l_w$.
\end{claim}

We choose a component of $y_1\cap Y$ non-isotopic to $l_w$, denoted by $l_y$. 
Similarly, we choose a component of $x_1\cap Y$ non-isotopic to $l_w$, denoted by $l_x$. 
Let $R$ denote the surface obtained by cutting $Y$ along $l_w$. 
The surface $R$ is homeomorphic to $S_{0, 4}$ and contains $\partial$ and $\partial_1$.
We denote by $\partial_3$ and $\partial_4$ the two boundary components of $R$ distinct from $\partial$ and $\partial_1$.

Let us say that two essential simple arcs $l_1$, $l_2$ in $Y$ {\it intersect minimally} if the cardinality of $l_1\cap l_2$ is equal to the minimal cardinality of $r_1\cap r_2$ among essential simple arcs $r_1$ and $r_2$ in $Y$ isotopic to $l_1$ and $l_2$, respectively.
Choose essential simple arcs $r_x$, $r_y$ in $Y$ such that
\begin{itemize}
\item $r_x$ is isotopic to $l_x$; $r_y$ is isotopic to $l_y$; and
\item any two of $r_x$, $r_y$ and $l_w$ intersect minimally.
\end{itemize}
We can find a curve $\bar{b}_1$ in $Y$ isotopic to $b_1$ and disjoint from $r_x$ and $l_w$.
Similarly, we can find a curve $\bar{a}_1$ in $Y$ isotopic to $a_1$ and disjoint from $r_y$ and $l_w$.
Since $l_x$ and $l_y$ are disjoint, $r_x$ and $r_y$ are disjoint.
When we cut $Y$ along $l_w$ and obtain $R$, the arc $r_y$ is decomposed into finitely many, essential simple arcs in $R$ each of which connects two distinct points of $\partial_3\cup \partial_4$. 
Let $L_y$ be the set of those essential simple arcs in $R$. 
In the same manner, we define the set $L_x$ after replacing $r_y$ with $r_x$.

We claim that each arc in $L_y$ connects $\partial_3$ and $\partial_4$ and that all arcs in $L_y$ are mutually isotopic as essential simple arcs in $R$. 
If there were an arc $s$ in $L_y$ connecting two points of $\partial_3$, then $s$ would be separating in $R$ and cut off a pair of pants containing $\partial$ and $\partial_1$ from $R$ because $s$ is disjoint from the curve $\bar{a}_1$ in $R$ cutting off a pair of pants containing $\partial$ and $\partial_1$ from $R$.
Let $u$ be an arc in $L_x$.
Since $r_x$ and $r_y$ are disjoint, either $u$ cuts off an annulus containing exactly one of $\partial$ and $\partial_1$ from $R$; $u$ is isotopic to $s$; or $u$ connects $\partial_3$ and $\partial_4$.
The first case is impossible because $r_x$ is disjoint from the curve $\bar{b}_1$ in $R$ cutting off a pair of pants containing $\partial$ and $\partial_1$ from $R$.
In the second and third cases, $\bar{a}_1$ and $\bar{b}_1$ has to be isotopic because $\bar{b}_1$ is disjoint from $u$. 
This is a contradiction.
A verbatim argument shows that $L_y$ contains no arc connecting two points of $\partial_4$. 
Any arc in $L_y$ thus connects $\partial_3$ and $\partial_4$. 
Since $r_y$ is disjoint from $\bar{a}_1$, all arcs in $L_y$ are mutually isotopic.

Similarly, we can show that any arc in $L_x$ connects $\partial_3$ and $\partial_4$, and all arcs in $L_x$ are mutually isotopic.
Pick $s_y\in L_y$ and $s_x\in L_x$.
The curve $a_1$ is isotopic to a boundary component of a regular neighborhood of $s_y\cup \partial_3\cup \partial_4$ in $R$.
Similarly, $b_1$ is isotopic to a boundary component of a regular neighborhood of $s_x\cup \partial_3\cup \partial_4$ in $R$.
The two arcs $s_y$ and $s_x$ are disjoint and non-isotopic because $a_1$ and $b_1$ are non-isotopic.
We can therefore find essential simple arcs $l_a$, $l_b$ in $R$ satisfying the conditions in Lemma \ref{lem-i4}.
\end{proof}

We define $b'$ and $w'$ as the two curves in $S$ described in Figure \ref{fig-bw}.
\begin{figure}
\begin{center}
\includegraphics[width=12cm]{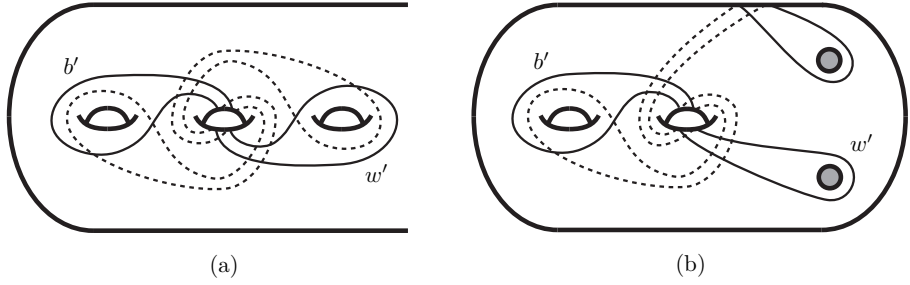}
\caption{(a) $S=S_{g, 1}$ with $g\geq 3$; (b) $S=S_{2, 2}$}\label{fig-bw}
\end{center}
\end{figure}
The argument so far can also be applied to the family of curves, $a$, $b'$, $x$, $y$, $z$ and $w'$, in place of $a$, $b$, $x$, $y$, $z$ and $w$. 
In particular, the following assertions hold:
\begin{itemize}
\item $\phi(b')$ is a p-BP in $S$ containing $c$.
We define a curve $b_1'$ in $S$ by the equality $\phi(b')=\{ c, b_1'\}$.
The curve $b_1'$ is a p-curve in $Y$ cutting off $\partial$ and $\partial_1$.
\item Since each of $\{ a, b'\}$ and $\{ b, b'\}$ is a sharing pair in $S$, there exists an essential simple arc $l_{b'}$ in $Y$ satisfying the following three conditions: $l_{b'}$ is disjoint from $b_1'$ and connects $\partial$ and $\partial_1$; $l_a$, $l_b$ and $l_{b'}$ are mutually disjoint and non-isotopic; and the surface obtained by cutting $Y$ along any two of $l_a$, $l_b$ and $l_{b'}$ is connected.
\item $\phi(w')$ is an h-curve in $S$ if $g\geq 3$, and is a p-BP in $S$ containing $c$ if $S=S_{2, 2}$. 
We put $w_1'=\phi(w')$ if $g\geq 3$, and define a curve $w_1'$ in $S$ by the equality $\phi(w')=\{ c, w_1'\}$ if $S=S_{2, 2}$.
\item The intersection $w_1'\cap Q$ consists of mutually isotopic, essential simple arcs in $Q$ which are non-separating in $Q$. 
Let $l_{w'}$ be a component of $w_1'\cap Q$.
\end{itemize}

\begin{lem}\label{lem-lwlwprime}
There exist essential simple arcs $r_w$, $r_{w'}$ in $Q$ such that $r_w$ is isotopic to $l_w$; $r_{w'}$ is isotopic to $l_{w'}$; $r_w$ and $r_{w'}$ are disjoint and non-isotopic; and the end points of $r_w$ and $r_{w'}$ appear alternately along the component of $\partial Q$ corresponding to $z_1$.
\end{lem}

\begin{proof}
We define $T_1$ as the surface obtained by cutting $Y$ along $l_a$.
Let $\partial_0$ denote the boundary component of $T_1$ consisting of four arcs corresponding to $\partial$, $\partial_1$ and two copies of $l_a$.
The arcs $l_b$ and $l_{b'}$ are essential simple arcs in $T_1$ connecting two points of $\partial_0$ and non-separating in $T_1$.

We claim that the end points of $l_b$ and $l_{b'}$ appear alternately along $\partial_0$.
Let $T_2$ denote the surface obtained by cutting $T_1$ along $l_b$, which is a pair of pants.
We have two components of $\partial T_2$ each of which consists of four arcs corresponding to $\partial$, $\partial_1$, $l_a$ and $l_b$ (see Figure \ref{fig-pants} (a)).
\begin{figure}
\begin{center}
\includegraphics[width=12cm]{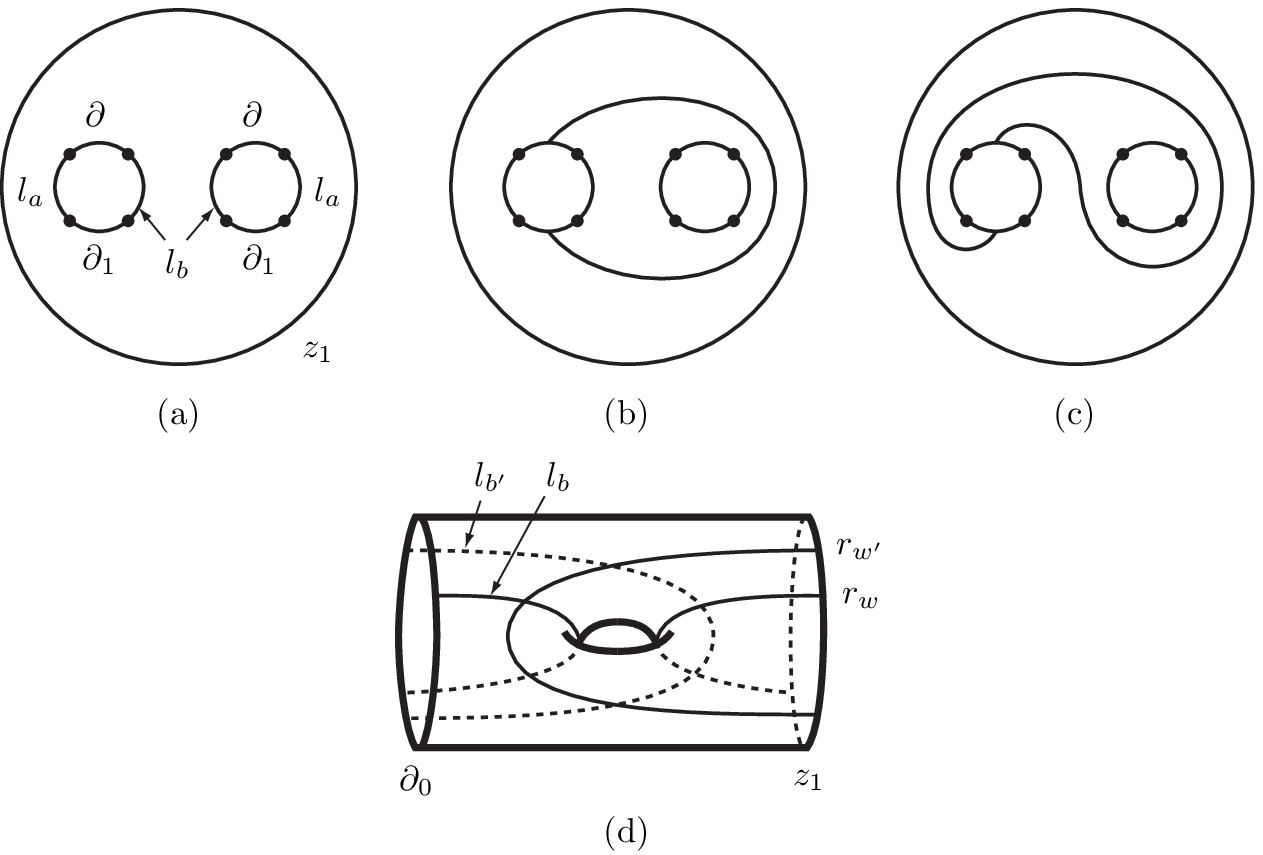}
\caption{}\label{fig-pants}
\end{center}
\end{figure}
If the claim were not true, then up to a homeomorphism of $Y$ fixing each of $\partial$, $\partial_1$, $l_a$ and $l_b$ as a set, we would have the two possibilities for $l_{b'}$ indicated in Figure \ref{fig-pants} (b) and (c).
In case (b), the surface obtained by cutting $Y$ along $l_a\cup l_{b'}$ is not connected.
In case (c), the surface obtained by cutting $Y$ along $l_b\cup l_{b'}$ is not connected.
We thus obtain a contradiction in both cases.

The claim shown in the last paragraph and Remark \ref{rem-arc} imply that $l_b$ and $l_{b'}$ are described as in Figure \ref{fig-pants} (d) up to a homeomorphism of $T_1$.
Since $T_2$ is a pair of pants, the arc $l_w$ is uniquely determined, up to isotopy, as an essential simple arc in $T_1$ which is disjoint from $l_b$ and connects two points of the boundary component of $T_1$ corresponding to $z_1$.
Similarly, up to isotopy, $l_{w'}$ is uniquely determined by $l_{b'}$.
We can thus find essential simple arcs $r_w$, $r_{w'}$ in $T_1$ isotopic to $l_w$, $l_{w'}$, respectively, and described as in Figure \ref{fig-pants} (d).
The lemma is proved.
\end{proof}

Let $X$ denote the component of $S_z$ containing $a$, which is homeomorphic to $S_{2, 1}$.
We note that each of $w\cap X$ and $w'\cap X$ consists of two isotopic, essential simple arcs in $X$ which are non-separating in $X$. 
We can find essential simple arcs $s_w$, $s_{w'}$ in $X$ such that $s_w$ is isotopic to a component of $w\cap X$; $s_{w'}$ is isotopic to a component of $w'\cap X$; $s_w$ and $s_{w'}$ are disjoint and non-isotopic; and the end points of $s_w$ and $s_{w'}$ appear alternately along the component of $\partial X$ corresponding to $z$.
Let $\cal{B}=\cal{B}(X, a; s_w, s_{w'})$ be the simplicial graph defined in Section \ref{subsec-b}.
Recall that the set of vertices of $\cal{B}$ is the union $B(s_w)\cup B(s_{w'})$.
The set $B(s_w)$ consists of all elements $\beta$ of $V(X)$ such that $\beta$ is an h-curve in $X$; $a$ and $\beta$ form a sharing pair in $X$; and a representative of $\beta$ is disjoint from $s_{w}$.
The last condition is equivalent to the condition that $\beta$ is disjoint from $w$ as a curve in $S$.
The set $B(s_{w'})$ is defined similarly after exchanging $s_w$ with $s_{w'}$.

Let $\cal{G}=\cal{G}(Y, a_1; r_w, r_{w'})$ be the simplicial graph defined in Section \ref{subsec-g}.
The set of vertices of $\cal{G}$ is the union $\Gamma(r_w)\cup \Gamma(r_{w'})$.
The set $\Gamma(r_w)$ consists of all elements $\gamma$ of $V(Y)$ such that $\gamma$ is a p-curve in $Y$ cutting off $\partial$ and $\partial_1$; $a_1$ and $\gamma$ form a sharing pair in $Y$; and a representative of $\gamma$ and $r_w$ are disjoint.
The last condition is equivalent to the condition that $\gamma$ is disjoint from $w_1$ as a curve in $S$ by Claim \ref{claim-lw}.
The set $\Gamma(r_{w'})$ is defined similarly after exchanging $r_w$ with $r_{w'}$.

We are now ready to deduce a contradiction on the assumption that $\phi(a)$ is a p-BP in $S$. 
Since for each $\beta \in B(s_w)$, $\phi(\beta)$ is a p-BP in $S$ consisting of $c$ and an element of $\Gamma(r_w)$ by Lemmas \ref{lem-phi-b-c} and \ref{lem-i4}, we have the map from $B(s_w)$ to $\Gamma(r_w)$ associated with $\phi$. 
Similarly, we have the map from $B(s_{w'})$ to $\Gamma(r_{w'})$ associated with $\phi$. 
By Lemma \ref{lem-i4}, $\phi$ induces a simplicial map from $\cal{B}$ to $\cal{G}$, which is injective because $\phi$ is injective.
This contradicts Lemma \ref{lem-b-g}.
We thus proved the following:

\begin{lem}\label{lem-h-not-pbp}
If $a$ is an h-curve in $S$, then $\phi(a)$ is not a p-BP in $S$.
\end{lem}

\subsection{Conclusion}\label{subsec-conc}

We first conclude the following:

\begin{thm}\label{thm-31}
Let $S=S_{g, 1}$ be a surface with $g\geq 3$ and $\phi \colon \calc_s(S)\rightarrow \calt(S)$ a superinjective map. 
Then the inclusion $\phi(\calc_s(S))\subset \calc_s(S)$ holds.
\end{thm}

\begin{proof}
Lemma \ref{lem-h-not-pbp} implies that $\phi$ sends each h-curve in $S$ to an h-curve in $S$ because there is no p-curve in $S$.
Let $\alpha$ be a separating curve in $S$ which is not an hp-curve in $S$.
Choose $g$ h-curves in $S$ forming a $g$-simplex of $\calc_s(S)$ together with $\alpha$.
The map $\phi$ sends each of those $g$ h-curves in $S$ to an h-curve in $S$.
It follows that $\phi(\alpha)$ is not a BP in $S$.
\end{proof}

We next conclude the following:

\begin{thm}\label{thm-22}
Let $S=S_{2, 2}$ be a surface and $\phi \colon \calc_s(S)\rightarrow \calt(S)$ a superinjective map. 
Then the inclusion $\phi(\calc_s(S))\subset \calc_s(S)$ holds.
\end{thm}

\begin{proof}
Lemmas \ref{lem-22-pre} and \ref{lem-h-not-pbp} imply that $\phi$ preserves hp-curves in $S$.
Once $\phi$ is shown to preserve h-curves in $S$, the theorem is verified along argument in the proof of Theorem \ref{thm-31}.
Assuming that there exists an h-curve $a$ in $S$ with $\phi(a)$ a p-curve in $S$, we deduce a contradiction.
The proof consists of verbatim arguments of the proof of Lemma \ref{lem-h-not-pbp}.
We thus give only a sketch of it.

We choose the curves $b$, $x$, $y$, $z$ and $w$ in $S$ described in Figure \ref{fig-s-pair-cha} (b).
Since $\phi$ preserves hp-curves in $S$ and since $\phi(a)$ is a p-curve in $S$, each of $\phi(y)$, $\phi(z)$ and $\phi(w)$ is an h-curve in $S$.
We denote by $\partial_1$ and $\partial_2$ the two components of $\partial S$.
Let $Q$ denote the component of $S_{\{ \phi(a), \phi(z)\}}$ homeomorphic to $S_{1, 2}$.
\begin{figure}
\begin{center}
\includegraphics[width=6cm]{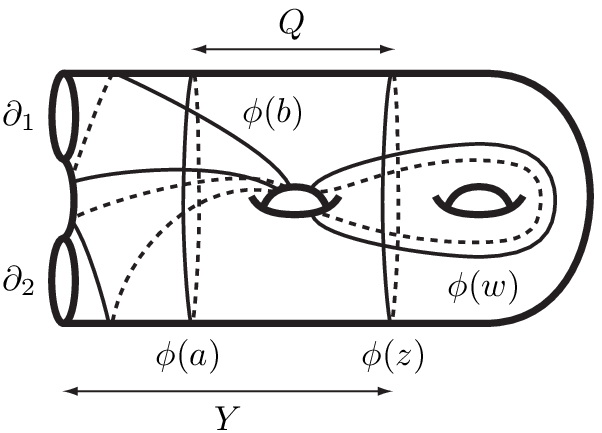}
\caption{}\label{fig-22-thm}
\end{center}
\end{figure}
Let $Y$ denote the component of $S_{\phi(z)}$ that is not a handle (see Figure \ref{fig-22-thm}).

Following the proof of Lemma \ref{lem-phib-not-h}, we can show that $\phi(b)$ is not an h-curve in $S$ and is thus a p-curve in $S$.
It turns out that $\phi(x)$ is an h-curve in $S$ because $\phi(x)$ is an hp-curve in $S$ disjoint from $\phi(b)$.
Following the proof of Lemma \ref{lem-i4}, we can show that
\begin{itemize}
\item $\phi(a)$ and $\phi(b)$ form a sharing pair in $Y$; and
\item $\phi(w)\cap Q$ consists of mutually isotopic, essential simple arcs in $Q$ which are non-separating in $Q$.
\end{itemize}
Choose the curves $b'$, $w'$ in $S$ described in Figure \ref{fig-bw} (b).
We can then show that
\begin{itemize}
\item $\phi(b')$ is a p-curve in $Y$ cutting off $\partial_1$ and $\partial_2$;
\item any two of $\phi(a)$, $\phi(b)$ and $\phi(b')$ form a sharing pair in $Y$;
\item $\phi(w')$ is an h-curve in $S$; and
\item $\phi(w')\cap Q$ consists of mutually isotopic, essential simple arcs in $Q$ which are non-separating in $Q$.
\end{itemize}
Following the proof of Lemma \ref{lem-lwlwprime}, we can find essential simple arcs $r_w$, $r_{w'}$ in $Q$ such that $r_w$ is isotopic to a component of $\phi(w)\cap Q$; $r_{w'}$ is isotopic to a component of $\phi(w')\cap Q$; $r_w$ and $r_{w'}$ are disjoint and non-isotopic; and the end points of $r_w$ and $r_{w'}$ appear alternately along the component of $\partial Q$ corresponding to $\phi(z)$.
Let $\cal{G}=\cal{G}(Y, \phi(a); r_w, r_{w'})$ be the graph defined in Section \ref{subsec-g}.

Let $X$ denote the component of $S_z$ containing $a$, which is homeomorphic to $S_{2, 1}$.
As discussed in the paragraph right after the proof of Lemma \ref{lem-lwlwprime}, we pick essential simple arcs $s_w$, $s_{w'}$ in $X$ and define the graph $\cal{B}=\cal{B}(X, a; s_w, s_{w'})$.
We then obtain the injective simplicial map from $\cal{B}$ into $\cal{G}$ associated with $\phi$.
This contradicts Lemma \ref{lem-b-g}.
\end{proof}


\section{$S_{1, 4}$}\label{sec-14}

We start with the following brief observation on hp-curves in $S_{1, 4}$ and their image via a superinjective map into the Torelli complex. 

\begin{lem}\label{lem-14-hp}
Let $S=S_{1, 4}$ be a surface and $\phi \colon \calc_s(S)\rightarrow \calt(S)$ a superinjective map. 
Choose hp-curves $\alpha$, $\beta$ and $\gamma$ in $S$ which are mutually disjoint and distinct. 
We assume that at least one of $\phi(\alpha)$, $\phi(\beta)$ and $\phi(\gamma)$ is a BP in $S$. 
Then exactly one of $\phi(\alpha)$, $\phi(\beta)$ and $\phi(\gamma)$ is a p-curve in $S$, and the others are p-BPs in $S$.
\end{lem}

\begin{proof}
By our assumption, each of $\phi(\alpha)$, $\phi(\beta)$ and $\phi(\gamma)$ is either a p-curve in $S$ or a p-BP in $S$. 
Suppose that $\phi(\alpha)$ is a p-BP in $S$. 
Since it is impossible that both $\phi(\beta)$ and $\phi(\gamma)$ are p-curves in $S$, one of them, say $\phi(\beta)$, is a p-BP in $S$. 
If $\phi(\gamma)$ were also a p-BP in $S$, then $\phi(\alpha)$, $\phi(\beta)$ and $\phi(\gamma)$ would be mutually BP-equivalent because $S$ is of genus one. 
This contradicts the fact that each BP-equivalence class in a weakly rooted simplex of $\calt(S)$ contains at most two p-BPs in $S$.
\end{proof}

Let $\alpha$, $\beta$ and $\gamma$ be hp-curves in $S$ which are mutually disjoint and distinct. 
Throughout Sections \ref{subsec-alpha-p} and \ref{subsec-beta-p}, we deduce a contradiction on the assumption that at least one of $\phi(\alpha)$, $\phi(\beta)$ and $\phi(\gamma)$ is a BP in $S$. 
We may assume that $\alpha$ is an h-curve in $S$ and both $\beta$ and $\gamma$ are p-curves in $S$. 
In Section \ref{subsec-alpha-p}, assuming that $\phi(\alpha)$ is a p-curve in $S$, we deduce a contradiction. 
In Section \ref{subsec-beta-p}, assuming that $\phi(\beta)$ is a p-curve in $S$, we deduce a contradiction. 
Once these are proved, we see that each of $\phi(\alpha)$, $\phi(\beta)$ and $\phi(\gamma)$ is an hp-curve in $S$ by Lemma \ref{lem-14-hp}. 
In Section \ref{subsec-conclusion}, we conclude that $\phi$ preserves separating curves in $S$.

Throughout the rest of this section, we put $S=S_{1, 4}$ and fix a superinjective map $\phi \colon \calc_s(S)\rightarrow \calt(S)$.

\subsection{The case where $\phi(\alpha)$ is a p-curve in $S$}\label{subsec-alpha-p}

Let $\alpha$ be an h-curve in $S$ and $\beta$, $\gamma$ p-curves in $S$ such that $\{ \alpha, \beta, \gamma \}$ is a 2-simplex of $\calc_s(S)$.
Assuming that $\phi(\alpha)$ is a p-curve in $S$ and that both $\phi(\beta)$ and $\phi(\gamma)$ are p-BPs in $S$, we deduce a contradiction.
\begin{figure}
\begin{center}
\includegraphics[width=11cm]{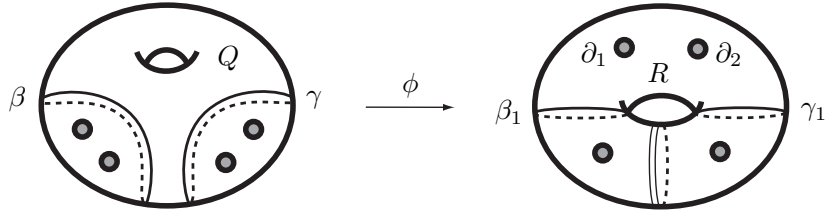}
\caption{The double line is the root curve for $\{ \phi(\beta), \phi(\gamma)\}$.}\label{fig-14-htop}
\end{center}
\end{figure}
Let $Q$ denote the component of $S_{\{ \beta, \gamma \}}$ homeomorphic to $S_{1, 2}$.
We denote by $\partial_1$ and $\partial_2$ the two components of $\partial S$ that are not contained in the pairs of pants cut off by $\phi(\beta)$ and $\phi(\gamma)$ from $S$. 
We define $R$ to be the component of $S_{\{ \phi(\beta), \phi(\gamma)\}}$ containing $\partial_1$ and $\partial_2$.
Let $\beta_1$ and $\gamma_1$ denote the curves of $\phi(\beta)$ and $\phi(\gamma)$, respectively, distinct from the root curve for $\{ \phi(\beta), \phi(\gamma)\}$ (see Figure \ref{fig-14-htop}).

\begin{lem}\label{lem-pre-h4-p4}
In the above notation, the following assertions hold:
\begin{enumerate}
\item For each h-curve $a$ in $Q$, $\phi(a)$ is a curve in $R$ cutting off a pair of pants containing $\partial_1$ and $\partial_2$ from $R$.
\item If $\alpha_1$ and $\alpha_2$ are h-curves in $Q$ with $i(\alpha_1, \alpha_2)=4$, then $i(\phi(\alpha_1), \phi(\alpha_2))=4$.
\end{enumerate}
\end{lem}

\begin{proof}
Assertion (i) follows from Lemma \ref{lem-14-hp}.
To prove assertion (ii), we choose the p-curves $\delta_1$, $\delta_2$ in $S$ described in Figure \ref{fig-14}.
We then have
\[i(\delta_1, \alpha_2)=i(\delta_2, \alpha_1)=i(\delta_1, \delta_2)=0,\]
\[i(\delta_j, \alpha_j)\neq 0,\quad i(\delta_j, \beta)\neq 0,\quad i(\delta_j, \gamma)\neq 0\quad \textrm{for\ any\ }j=1, 2.\]

Let $l$ be an essential simple arc in $R$ connecting two points of $\beta_1\cup \gamma_1$. 
We note that up to isotopy, there exists exactly one curve in $R$ disjoint from $l$.
If the curve in $R$ cuts off a pair of pants containing $\partial_1$ and $\partial_2$ from $R$, then either
\begin{enumerate} 
\item[(1)] $l$ is a separating arc in $R$ connecting two points of $\beta_1$ and cutting off from $R$ an annulus containing $\gamma_1$ as its boundary component;
\item[(2)] $l$ is a separating arc in $R$ connecting two points of $\gamma_1$ and cutting off from $R$ an annulus containing $\beta_1$ as its boundary component; or
\item[(3)] $l$ is an arc connecting a point of $\beta_1$ with a point of $\gamma_1$.
\end{enumerate}
We denote by $P$ the component of $R_{\phi(\alpha_1)}$ containing $\beta_1$ and $\gamma_1$ as its boundary components.
Since $P$ is a pair of pants, the intersection $\phi(\alpha_2)\cap P$ consists of mutually isotopic, essential simple arcs in $P$ connecting two points of the component of $\partial P$ corresponding to $\phi(\alpha_1)$.
It follows that $\phi(\delta_2)$ intersects $\phi(\alpha_2)\cap P$ because $\phi(\delta_2)$ is disjoint from $\phi(\alpha_1)$ and intersects $\phi(\alpha_2)$.
The intersection $\phi(\delta_2)\cap R$ is thus non-empty and contains an essential simple arc in $R$, denoted by $l_2$, connecting two points of $\beta_1\cup \gamma_1$.
Similarly, the same property holds for $\phi(\delta_1)\cap R$.
We pick an essential simple arc in $R$ contained in $\phi(\delta_1)\cap R$ and denote it by $l_1$.
\begin{figure}
\begin{center}
\includegraphics[width=6cm]{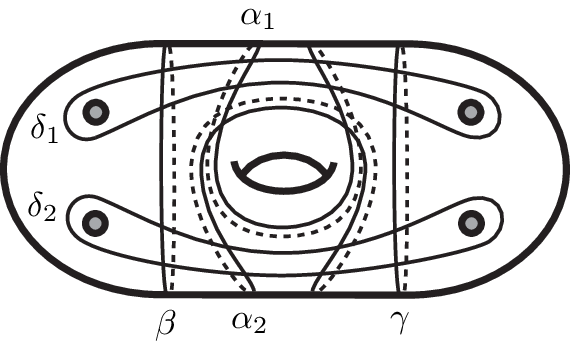}
\caption{}\label{fig-14}
\end{center}
\end{figure}
The existence of $\phi(\alpha_1)$ and $\phi(\alpha_2)$ implies that each of $l_1$ and $l_2$ satisfies one of conditions (1), (2) and (3). 
We note that if an essential simple arc $r_1$ in $R$ satisfying one of (1), (2) and (3) is disjoint from an essential simple arc $r_2$ in $R$ satisfying either (1) or (2), then the curve in $R$ disjoint from $r_1$ and the curve in $R$ disjoint from $r_2$ are isotopic. 
It follows that if there were $j\in \{ 1, 2\}$ such that $l_j$ satisfies either (1) or (2), then we would have the equality $\phi(\alpha_1)=\phi(\alpha_2)$.
This is a contradiction. 
Each of $l_1$ and $l_2$ therefore connects a point of $\beta_1$ with a point of $\gamma_1$.
Since $\phi(\delta_1)$ and $\phi(\delta_2)$ are disjoint, so are $l_1$ and $l_2$.
Since $\phi(\alpha_1)$ and $\phi(\alpha_2)$ intersect, $l_1$ and $l_2$ are not isotopic.
By Proposition \ref{prop-diag}, we have $i(\phi(\alpha_1), \phi(\alpha_2))=4$.
Assertion (ii) follows.
\end{proof}

We define $\cal{D}=\cal{D}(Q)$ and $\cal{E}=\cal{E}(R; \partial_1, \partial_2)$ to be the simplicial graphs introduced in Sections \ref{subsec-d} and \ref{subsec-farey}, respectively.
It follows from Lemma \ref{lem-pre-h4-p4} that $\phi$ induces an injective simplicial map from $\cal{D}$ into $\cal{E}$. 
This contradicts Lemma \ref{lem-d-e}.
We therefore obtain the following:

\begin{lem}\label{lem-14-hhp}
Let $S=S_{1, 4}$ be a surface and $\phi \colon \calc_s(S)\rightarrow \calt(S)$ a superinjective map. 
Choose an h-curve $\alpha$ in $S$ and p-curves $\beta$, $\gamma$ in $S$ such that $\{ \alpha, \beta, \gamma \}$ is a 2-simplex of $\calc_s(S)$. 
Then it is impossible that $\phi(\alpha)$ is a p-curve in $S$ and each of $\phi(\beta)$ and $\phi(\gamma)$ is a p-BP in $S$.
\end{lem}


\subsection{The case where $\phi(\beta)$ is a p-curve in $S$}\label{subsec-beta-p}

We fix two disjoint and distinct p-curves $\beta$, $\gamma$ in $S$. 
Throughout this subsection, we deduce a contradiction on the assumption that $\phi(\beta)$ is a p-curve in $S$ and that $\phi(\gamma)$ is a p-BP in $S$.
By Lemma \ref{lem-14-hp}, $\phi$ sends each h-curve in $S$ disjoint from $\beta$ and $\gamma$ to a p-BP in $S$ containing a curve of $\phi(\gamma)$. 
Let us introduce symbols employed in this subsection.

Let $\gamma_1$ and $\gamma_2$ denote the two curves of $\phi(\gamma)$. 
We denote by $Q$ the component of $S_{\{ \beta, \gamma \}}$ homeomorphic to $S_{1, 2}$ and denote by $R$ the component of $S_{\{ \phi(\beta), \phi(\gamma)\}}$ homeomorphic to $S_{0, 4}$. 
Let $\partial$ denote the component of $\partial S$ contained in $R$ (see Figure \ref{fig-14-ptop}).
For each h-curve $\alpha$ in $S$ disjoint from $\beta$ and $\gamma$, we define $c(\alpha)\in V(R)$ to be the curve of the p-BP $\phi(\alpha)$ that is not contained in $\phi(\gamma)$.
We define $\cal{F}$ as the simplicial graph $\cal{F}(R)$ introduced in Section \ref{subsec-farey}.

\begin{lem}\label{lem-pre-h2-diag}
In the above notation, the following assertions hold:
\begin{enumerate}
\item For each h-curve $\alpha$ in $Q$, the two components of $\partial R$ corresponding to $\gamma_1$ and $\gamma_2$ are contained in distinct components of $R_{c(\alpha)}$.
\item If $\alpha_1$ and $\alpha_2$ are h-curves in $Q$ with $i(\alpha_1, \alpha_2)=4$, then either we have $i(c(\alpha_1), c(\alpha_2))=2$ or $c(\alpha_1)$ and $c(\alpha_2)$ lie in a diagonal position of two adjacent triangles in $\cal{F}$.
\end{enumerate}
\end{lem}
 
\begin{proof}
Assertion (i) follows because for each h-curve $\alpha$ in $Q$, the curve $c(\alpha)$ is non-separating in $S$.
We prove assertion (ii).
Pick two h-curves $\alpha_1$, $\alpha_2$ in $Q$ with $i(\alpha_1, \alpha_2)=4$. 
As in the proof of Lemma \ref{lem-pre-h4-p4}, we choose the p-curves $\delta_1$, $\delta_2$ in $S$ described in Figure \ref{fig-14}.
We then have
\[i(\delta_1, \alpha_2)=i(\delta_2, \alpha_1)=i(\delta_1, \delta_2)=0,\]
\[i(\delta_j, \alpha_j)\neq 0,\quad i(\delta_j, \beta)\neq 0,\quad i(\delta_j, \gamma)\neq 0\quad \textrm{for\ any\ }j=1, 2.\]
For each $j=1, 2$, we choose a component $l_j$ of $\phi(\delta_j)\cap R$ at least one of whose end points lies in $\phi(\beta)$. 
The arc $l_j$ is simple and essential in $R$.
We list several properties of $l_1$ and $l_2$ below.
\begin{itemize}
\item Each $l_j$ connects a point of $\phi(\beta)$ with a point of $\phi(\beta)\cup \gamma_1\cup \gamma_2$. 
For each $j$ mod $2$, up to isotopy, $c(\alpha_j)$ is a unique curve in $R$ disjoint from $l_{j+1}$.
\item The arcs $l_1$ and $l_2$ are not isotopic because otherwise $c(\alpha_1)$ and $c(\alpha_2)$ would be isotopic, and this would contradict $i(\phi(\alpha_1), \phi(\alpha_2))\neq 0$.
\item The arcs $l_1$ and $l_2$ are disjoint because $\phi(\delta_1)$ and $\phi(\delta_2)$ are disjoint.
\item For each $j$ mod $2$, if both end points of $l_j$ lie in $\phi(\beta)$, then $l_j$ cuts off an annulus containing exactly one of $\gamma_1$ and $\gamma_2$. 
For otherwise $l_j$ would cut off an annulus containing $\partial$, and $c(\alpha_{j+1})$ would be separating in $S$.
\end{itemize}
For each $j=1, 2$, if $l_j$ connects a point of $\phi(\beta)$ with a point of $\gamma_1\cup \gamma_2$, then we set $r_j=l_j$. 
If $l_j$ connects two points of $\phi(\beta)$, then we define $r_j$ to be an essential simple arc in $R$ disjoint from $l_j$ and connecting a point of $\phi(\beta)$ with a point of the component of $\gamma_1\cup \gamma_2$ contained in the annulus cut off by $l_j$ from $R$. 
This arc is uniquely determined up to isotopy.
\begin{figure}
\begin{center}
\includegraphics[width=11cm]{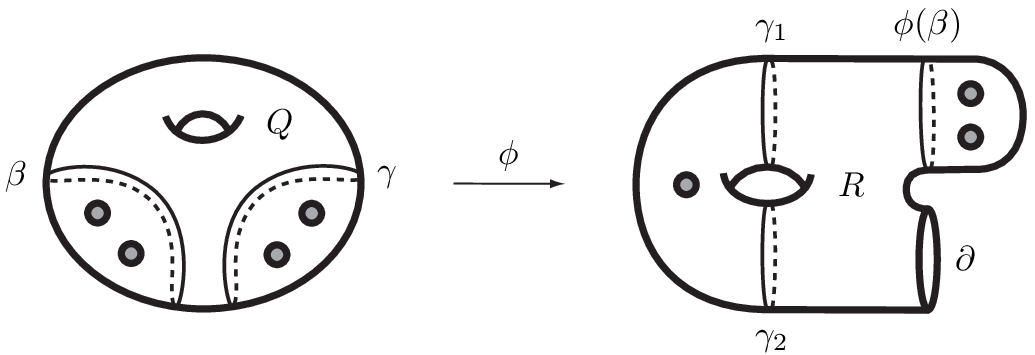}
\caption{}\label{fig-14-ptop}
\end{center}
\end{figure}

Let us say that two essential simple arcs $s_1$ and $s_2$ in $R$ are {\it disjoint} if $s_1$ and $s_2$ can be isotoped so that they are disjoint.
For each essential simple arc $s$ in $R$, if $s$ is disjoint from $l_j$, then $s$ is disjoint from $r_j$. 
It follows that $r_1$ and $r_2$ are disjoint.

If the end points of $r_1$ and $r_2$ that are not in $\phi(\beta)$ lie in distinct components of $\gamma_1\cup \gamma_2$, then the equality $i(c(\alpha_1), c(\alpha_2))=2$ follows. 
If all of those points lie in $\gamma_k$ for some $k\in \{ 1, 2\}$, then $c(\alpha_1)$ and $c(\alpha_2)$ lie in a diagonal position of two adjacent triangles in $\cal{F}$ by Proposition \ref{prop-diag}.
Assertion (ii) is proved.
\end{proof}

We now deduce a contradiction.
Let $\cal{D}=\cal{D}(Q)$ be the graph defined in Section \ref{subsec-d}.
Let $\cal{H}=\cal{H}(R; \gamma_1, \gamma_2)$ be the graph defined in Section \ref{subsec-farey}.
If $\alpha_1$ and $\alpha_2$ are distinct h-curves in $Q$, then we have $i(\alpha_1, \alpha_2)\neq 0$ and thus $i(\phi(\alpha_1), \phi(\alpha_2))\neq 0$. 
Since each of $\phi(\alpha_1)$ and $\phi(\alpha_2)$ is a p-BP in $S$ containing a curve in $\phi(\gamma)$, we have $i(c(\alpha_1), c(\alpha_2))\neq 0$.
In particular, we have $c(\alpha_1)\neq c(\alpha_2)$. 
The map $c$ from the set of h-curves in $Q$ into the set of curves in $R$ is therefore injective and induces a simplicial map from $\cal{D}$ into $\cal{H}$ by Lemma \ref{lem-pre-h2-diag}.
This contradicts Lemma \ref{lem-d-h}.
We thus obtain the following:

\begin{lem}\label{lem-14-php}
Let $S=S_{1, 4}$ be a surface and $\phi \colon \calc_s(S)\rightarrow \calt(S)$ a superinjective map. 
Choose an h-curve $\alpha$ in $S$ and p-curves $\beta$, $\gamma$ in $S$ such that $\{ \alpha, \beta, \gamma \}$ is a 2-simplex of $\calc_s(S)$. 
Then it is impossible that $\phi(\beta)$ is a p-curve in $S$ and each of $\phi(\alpha)$ and $\phi(\gamma)$ is a p-BP in $S$.
\end{lem}


\subsection{Conclusion}\label{subsec-conclusion}

We conclude the following:

\begin{thm}\label{thm-14}
Let $S=S_{1, 4}$ be a surface and $\phi \colon \calc_s(S)\rightarrow \calt(S)$ a superinjective map. 
Then the inclusion $\phi(\calc_s(S))\subset \calc_s(S)$ holds.
\end{thm}

\begin{proof}
As discussed in the beginning of this section, Lemmas \ref{lem-14-hp}, \ref{lem-14-hhp} and \ref{lem-14-php} imply that $\phi$ preserves hp-curve in $S$. 
Once $\phi$ is shown to preserve h-curves in $S$, it turns out that for any separating curve $a$ in $S$ that is not an hp-curve in $S$, $\phi(a)$ is not a BP in $S$.
For there exists an h-curve $b$ in $S$ disjoint from $a$, and $\phi(a)$ is disjoint from the h-curve $\phi(b)$ in $S$.
The theorem thus follows.

Assuming that there exists an h-curve $\alpha$ in $S$ with $\phi(\alpha)$ a p-curve in $S$, we deduce a contradiction. 
Choose p-curves $\beta$, $\gamma$ in $S$ such that $\{ \alpha, \beta, \gamma \}$ is a 2-simplex of $\calc_s(S)$. 
We note that one of $\phi(\beta)$ and $\phi(\gamma)$ is an h-curve in $S$ and another is a p-curve in $S$. 
Let $Q$ denote the component of $S_{\{ \beta, \gamma \}}$ homeomorphic to $S_{1, 2}$.
Let $R$ denote the component of $S_{\{ \phi(\beta), \phi(\gamma)\}}$ homeomorphic to $S_{0, 4}$. 
We denote by $\partial_1$ and $\partial_2$ the two components of $\partial S$ contained in $R$.
For each h-curve $a$ in $Q$, $\phi(a)$ is a curve in $R$ cutting off a pair of pants containing $\partial_1$ and $\partial_2$ from $R$.

In the proof of Lemma \ref{lem-pre-h4-p4} (ii), after replacing $\beta_1$ and $\gamma_1$ with $\phi(\beta)$ and $\phi(\gamma)$, respectively, a verbatim argument shows that for any two h-curves $\alpha_1$, $\alpha_2$ in $Q$ with $i(\alpha_1, \alpha_2)=4$, we have $i(\phi(\alpha_1), \phi(\alpha_2))=4$.
It follows that $\phi$ induces an injective simplicial map from the graph $\cal{D}(Q)$ into the graph $\cal{E}(R; \partial_1, \partial_2)$.
These two graphs are defined in Sections \ref{subsec-d} and \ref{subsec-farey}, respectively.
Since such an injective simplicial map does not exist by Lemma \ref{lem-d-e}, we obtain a contradiction.
\end{proof}


\section{The other surfaces}\label{sec-other}

In this section, we discuss the remainder of surfaces.
Let $S$ be a surface and pick two distinct components $\partial_1$, $\partial_2$ of $\partial S$.
For each vertex $a$ of $\calt(S)$, we say that $a$ {\it separates $\partial_1$ and $\partial_2$} if $\partial_1$ and $\partial_2$ are contained in distinct components of $S_a$.

\begin{thm}\label{thm-other}
Let $S=S_{g, p}$ be a surface and assume one of the following three conditions: $g\geq 3$ and $p\geq 1$; $g=2$ and $p\geq 2$; or $g=1$ and $p\geq 4$.
If $\phi \colon \calc_s(S)\rightarrow \calt(S)$ is a superinjective map, then $\phi$ is induced by an element of $\mod^*(S)$.
\end{thm}

\begin{proof}
We prove the theorem by induction on $p$. 
In the case where $S$ is either $S_{g, 1}$ with $g\geq 3$, $S_{2, 2}$ or $S_{1, 4}$, the conclusion follows from Theorems \ref{thm-cs}, \ref{thm-31}, \ref{thm-22} and \ref{thm-14}.
In what follows, we assume that $S$ is distinct from these surfaces.

\begin{lem}\label{lem-p-hp}
The map $\phi$ sends each p-curve in $S$ to an hp-curve in $S$.
\end{lem}

\begin{proof}
Assuming that there is a p-curve $\alpha$ in $S$ with $\phi(\alpha)$ a p-BP in $S$, we deduce a contradiction. 
Let $Q$ denote the component of $S_{\alpha}$ that is not a pair of pants.
Let $\partial$ denote the component of $\partial S$ cut off by $\phi(\alpha)$.
We define $R$ as the surface obtained from $S$ by attaching a disk to $\partial$.
Let $\calc^*(R)$ be the simplicial cone over $\calc(R)$ with $\ast$ the cone point.
We define $\imath \colon \calc(S)\rightarrow \calc^*(R)$ as the simplicial map associated with the inclusion of $S$ into $R$, where $\imath^{-1}(\{ \ast \})$ consists of all p-curves in $S$ cutting off $\partial$.
The map $\imath$ sends the two curves in $\phi(\alpha)$ to the same non-separating curve in $R$, denoted by $c$.

We define a superinjective map $\phi_{\alpha}\colon \calc_s(Q)\rightarrow \calt(R)$ as follows.
Pick a separating curve $\beta$ in $S$ disjoint and distinct from $\alpha$.
If $\phi(\beta)$ is a separating curve in $S$, then $\imath(\phi(\beta))$ is a separating curve in $R$, and we set $\phi_{\alpha}(\beta)=\imath(\phi(\beta))$.
If $\phi(\beta)$ is a BP in $S$ containing no curve in $\phi(\alpha)$, then $\phi(\beta)$ is not BP-equivalent to $\phi(\alpha)$ by Lemma \ref{lem-si-bp}, and thus $\imath(\phi(\beta))$ is a BP in $R$.
We set $\phi_{\alpha}(\beta)=\imath(\phi(\beta))$.
If $\phi(\beta)$ is a BP in $S$ containing a curve in $\phi(\alpha)$, then we denote by $\beta_1$ another curve in $\phi(\beta)$.
The pair $\{ c, \imath(\beta_1)\}$ is a BP in $R$, and we set $\phi_{\alpha}(\beta)=\{ c, \imath(\beta_1)\}$.
The map $\phi_{\alpha}\colon V_s(Q)\rightarrow V_t(R)$ then defines a superinjective map from $\calc_s(Q)$ into $\calt(R)$. 

Since both $Q$ and $R$ are homeomorphic to $S_{g, p-1}$, the hypothesis of the induction implies that $\phi_{\alpha}$ is induced by a homeomorphism from $Q$ onto $R$. 
For any $\gamma \in V_s(Q)$, $\phi_{\alpha}(\gamma)$ is however disjoint from $c$.
This is a contradiction.
\end{proof}

\begin{lem}\label{lem-hp-hp}
The map $\phi$ sends each hp-curve in $S$ to an hp-curve in $S$.
\end{lem}

\begin{proof}
Let $\alpha$ be an h-curve in $S$.
By Lemma \ref{lem-p-hp}, it is enough to show that $\phi(\alpha)$ is an hp-curve in $S$.
Choose the p-curve $\beta$ in $S$ and the separating curves $\beta'$, $\gamma$ in $S$ described in Figure \ref{fig-hp-pre} (a).
\begin{figure}
\begin{center}
\includegraphics[width=11cm]{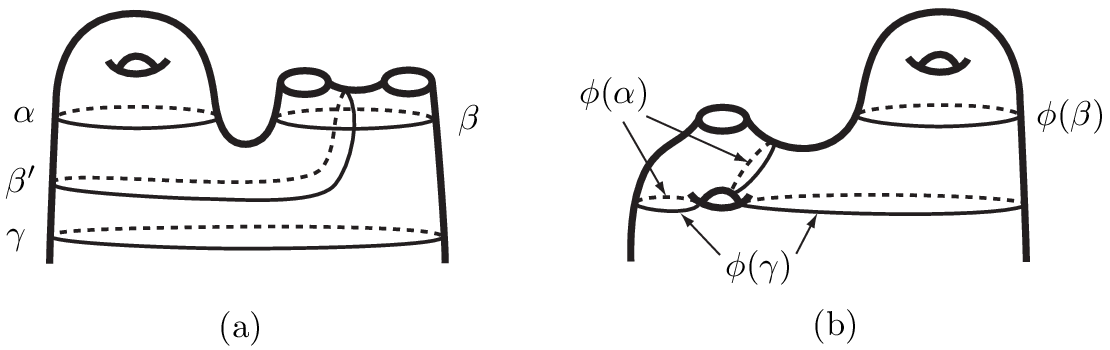}
\caption{}\label{fig-hp-pre}
\end{center}
\end{figure}

We first suppose that $\phi(\beta)$ is a p-curve in $S$.
Let $Q_1$ denote the component of $S_{\beta}$ that is not a pair of pants.
Let $Q_2$ denote the component of $S_{\phi(\beta)}$ that is not a pair of pants.
The map $\phi$ induces a superinjective map from $\calc_s(Q_1)$ into $\calt(Q_2)$. 
The hypothesis of the induction then implies that $\phi(\alpha)$ is an h-curve in $S$.

We next suppose that $\phi(\beta)$ is not a p-curve in $S$ and is thus an h-curve in $S$ by Lemma \ref{lem-p-hp}. 
If $\phi(\alpha)$ were a p-BP in $S$, then $\phi(\gamma)$ would be a BP in $S$ and BP-equivalent to $\phi(\alpha)$ because $\phi$ is $\chi$-preserving (see Figure \ref{fig-hp-pre} (b)). 
We deduce a contradiction by proving that there is no room for $\phi(\beta')$. 
Let $R$ be the component of $S_{\phi(\alpha)\cup \phi(\gamma)}$ containing $\phi(\beta)$. 
If $\phi(\beta')$ is a separating curve in $S$, then $\phi(\beta')$ has to be an h-curve in $S$ because $\phi(\beta')$ lies in $R$. 
This is a contradiction because $\phi$ is $\chi$-preserving.
If $\phi(\beta')$ is a BP in $S$ and BP-equivalent to $\phi(\alpha)$, then $\{ \phi(\alpha), \phi(\beta'), \phi(\gamma)\}$ is rooted by Lemma \ref{lem-si-bp}. 
This is a contradiction because any separating curve in $R$ is an h-curve in $S$.
If $\phi(\beta')$ is a BP in $S$ and not BP-equivalent to $\phi(\alpha)$, then $\phi(\beta')$ is a BP in $R$.
This is a contradiction because no BP in $R$ is a BP in $S$.
\end{proof}

We first prove the theorem when $g\geq 3$ and $p=2$, which is a direct consequence of Theorem \ref{thm-cs} and the following:

\begin{lem}
If we have $S=S_{g, 2}$ with $g\geq 3$, then $\phi$ sends each separating curve in $S$ to a separating curve in $S$.
\end{lem}

\begin{proof}
Once we prove that $\phi$ sends each h-curve in $S$ to an h-curve in $S$, we can deduce the lemma along argument in the proof of Theorem \ref{thm-31}.
Assume that there exists an h-curve $\alpha$ in $S$ with $\phi(\alpha)$ a p-curve in $S$.
We choose a $g$-simplex $\sigma$ of $\calc_s(S)$ consisting of hp-curves in $S$ and containing $\alpha$.
Since $\phi(\sigma)$ also consists of hp-curves in $S$, if $\beta$ denotes the p-curve in $\sigma$, then $\phi(\beta)$ is an h-curve in $S$. 
Let $Q$ denote the component of $S_{\beta}$ that is not a pair of pants.
Let $R$ denote the component of $S_{\phi(\beta)}$ that is not a handle. 

We claim that for each separating curve $\gamma$ in $Q$ which is not an h-curve in $Q$, $\phi(\gamma)$ is not a BP in $S$.
Choose a $(g-1)$-simplex $\tau$ of $\calc_s(Q)$ consisting of h-curves in $Q$ and disjoint from $\gamma$. 
By Lemma \ref{lem-hp-hp}, $\phi(\tau)$ consists of $g-1$ h-curves in $R$ and one p-curve in $R$.
It turns out that $\phi(\gamma)$ is not a BP in $S$ because $\phi(\gamma)$ is disjoint from the $g-1$ h-curves in $\phi(\tau)$ and the h-curve $\phi(\beta)$ in $S$.
The claim is proved.

The claim shown in the last paragraph implies that $\phi$ induces a superinjective map $\phi_{\beta}\colon \calc_s(Q)\rightarrow \calc_s(R)$.
Let $\partial_1$ denote the component of $\partial R$ corresponding to $\phi(\beta)$.
Let $\partial_2$ denote another component of $\partial R$, which is a component of $\partial S$.
Gluing $\partial_1$ with $\partial_2$, we obtain the surface, denoted by $R_1$, homeomorphic to $S_{g, 1}$. 
Let $c$ denote the non-separating curve in $R_1$ corresponding to $\partial_1$ and $\partial_2$.
We denote by $\imath \colon \calc(R)\rightarrow \calc(R_1)$ the simplicial map associated with the inclusion of $R$ onto the complement of a tubular neighborhood of $c$ in $R_1$.

We define a superinjective map $\bar{\imath}\colon \calc_s(R)\rightarrow \calt(R_1)$ as follows.
Pick a separating curve $\delta$ in $R$.
If $\delta$ separates $\partial_1$ and $\partial_2$, then we define $\bar{\imath}(\delta)$ to be the BP $\{ c, \imath(\delta)\}$ in $R_1$.
Otherwise, $\imath(\delta)$ is a separating curve in $R_1$, and we set $\bar{\imath}(\delta)=\imath(\delta)$.
The map $\bar{\imath}\colon V_s(R)\rightarrow V_t(R_1)$ then defines a superinjective map from $\calc_s(R)$ into $\calt(R_1)$.

Let $\phi_1\colon \calc_s(Q)\rightarrow \calt(R_1)$ be the superinjective map defined as the composition $\bar{\imath}\circ \phi_{\beta}$. 
Since both $Q$ and $R_1$ are homeomorphic to $S_{g, 1}$, we have the inclusion $\phi_1(\calc_s(Q))\subset \calc_s(R_1)$ by Theorem \ref{thm-31}. 
On the other hand, by the definition of $\phi_1$, we have $\phi_1(\alpha)=\{ c, \imath(\phi(\alpha))\}$, which is a BP in $R_1$. 
This is a contradiction.  
\end{proof}

Finally, we assume either $g\geq 2$ and $p\geq 3$ or $g=1$ and $p\geq 5$.
Let us say that a separating curve $\alpha$ in $S$ is {\it even} if at least one component of $S_{\alpha}$ contains an even number of components of $\partial S$. 
Otherwise, $\alpha$ is said to be {\it odd}.
We note that if there exists an odd separating curve in $S$, then $p$ is even.

\begin{lem}\label{lem-even}
If $\alpha$ is an even separating curve in $S$, then $\phi(\alpha)$ is a separating curve in $S$.
\end{lem}

\begin{proof}
By Lemma \ref{lem-hp-hp}, we may assume that $\alpha$ is not an hp-curve in $S$. 
Since $\alpha$ is even, there exist $g+\lfloor p/2\rfloor$ hp-curves in $S$ forming a $(g+\lfloor p/2\rfloor)$-simplex of $\calc_s(S)$ together with $\alpha$.
Since the image of those hp-curves via $\phi$ contains $g$ h-curves in $S$ by Lemma \ref{lem-hp-hp}, $\phi(\alpha)$ is not a BP in $S$.
\end{proof}

\begin{lem}\label{lem-odd}
We assume that $p$ is even and $p\geq 4$. 
If $\alpha$ is an odd separating curve in $S$, then $\phi(\alpha)$ is a separating curve in $S$.
\end{lem}

\begin{proof}
Assume that there exists an odd separating curve $\alpha$ in $S$ with $\phi(\alpha)$ a BP in $S$.
We first prove the following:

\begin{claim}\label{claim-odd}
If $\beta$ is an odd separating curve in $S$ disjoint and distinct from $\alpha$, then $\phi(\beta)$ is a BP in $S$ and BP-equivalent to $\phi(\alpha)$.
\end{claim}

\begin{proof}
Let $\beta$ be an odd separating curve in $S$ disjoint and distinct from $\alpha$. 
We choose even separating curves $\alpha_1$, $\beta_1$, $\gamma$ and $\gamma_1$ in $S$ as described in Figure \ref{fig-sep-curves} (a).
\begin{figure}
\begin{center}
\includegraphics[width=12cm]{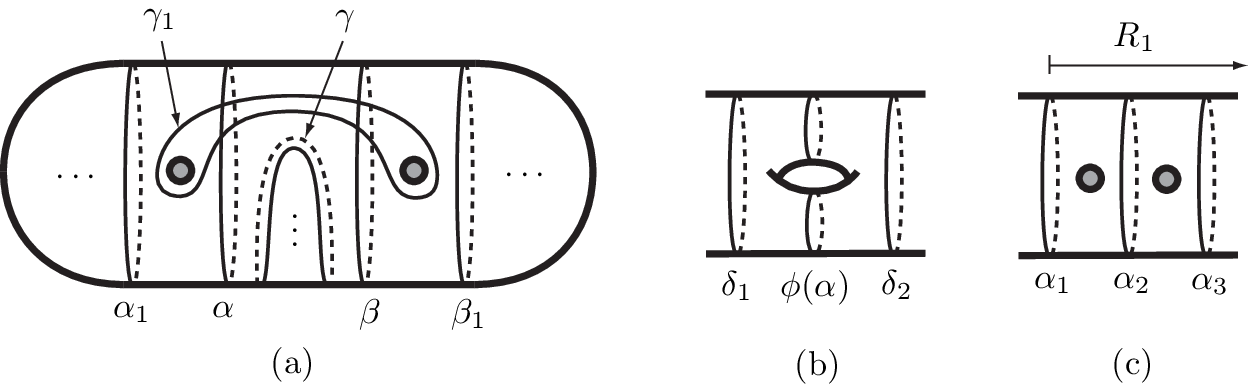}
\caption{}\label{fig-sep-curves}
\end{center}
\end{figure}
There then exists a simplex $\sigma$ of $\calc_s(S)$ of maximal dimension such that
\begin{itemize}
\item $\sigma$ contains $\alpha$, $\alpha_1$, $\beta$, $\beta_1$ and $\gamma$; and
\item any curve of $\sigma$ other than $\alpha$ and $\beta$ is even.
\end{itemize}
By Lemma \ref{lem-even}, any element of $\phi(\sigma)\setminus \{ \phi(\alpha), \phi(\beta)\}$ is a separating curve in $S$.
Each component of $S_{\phi(\sigma)}$ is either a handle or a pair of pants by Lemma \ref{lem-w-rooted} (ii).
Note that $\phi(\alpha)$ is not a p-BP in $S$ because $\alpha$ is not an hp-curve in $S$.

Assuming that $\phi(\beta)$ is a separating curve in $S$, we deduce a contradiction.
Since any element of $\phi(\sigma)$ other than $\phi(\alpha)$ is then a separating curve in $S$, there exist exactly two components of $S_{\phi(\sigma)}$ each of whose boundary components contains both curves of $\phi(\alpha)$ (see Figure \ref{fig-sep-curves} (b)). 
The number of curves $\delta$ of $\phi(\sigma)\setminus \{ \phi(\alpha)\}$ satisfying the following condition is thus equal to two: There exists a vertex $\epsilon \in V_t(S)$ with $i(\phi(\alpha), \epsilon)\neq 0$, $i(\delta, \epsilon)\neq 0$ and $i(\delta', \epsilon)=0$ for any $\delta'\in \phi(\sigma)\setminus \{ \phi(\alpha), \delta\}$. 
On the other hand, each of $\phi(\alpha_1)$, $\phi(\beta)$ and $\phi(\gamma)$ satisfies this condition because for each of $\alpha_1$, $\beta$ and $\gamma$, there exists a component of $S_{\sigma}$ containing it and $\alpha$ as boundary components. 
This is a contradiction.
It turns out that $\phi(\beta)$ is a BP in $S$.

If $\phi(\beta)$ were not BP-equivalent to $\phi(\alpha)$, then there would exist a separating curve $\delta$ in $\phi(\sigma)$ such that $\phi(\alpha)$ and $\phi(\beta)$ are contained in distinct components of $S_{\delta}$. 
On the other hand, the hp-curve $\phi(\gamma_1)$ intersects $\phi(\alpha)$ and $\phi(\beta)$ and is disjoint from the other elements of $\phi(\sigma)$. 
This is impossible.
\end{proof}

We return to the proof of Lemma \ref{lem-odd}. 
Choose a component of $S_{\alpha}$ containing at least three components of $\partial S$ and denote it by $R$. 
If $R$ is homeomorphic to $S_{0, 4}$, then pick an odd separating curve $\alpha_1$ in $S$ cutting off a surface, denoted by $R_1$, containing $R$ and homeomorphic to $S_{1, 4}$. 
If $R$ is not homeomorphic to $S_{0, 4}$, then we put $\alpha_1=\alpha$ and $R_1=R$. 
We can then choose separating curves $\alpha_2$, $\alpha_3$ in $R_1$ as described in Figure \ref{fig-sep-curves} (c). 
The curves $\alpha_1$ and $\alpha_3$ are odd, and $\alpha_2$ is even. 
By Claim \ref{claim-odd}, $\phi(\alpha_1)$ is a BP in $S$. 
Applying Claim \ref{claim-odd} to $\alpha_1$ in place of $\alpha$, we see that $\phi(\alpha_3)$ is also a BP in $S$ and BP-equivalent to $\phi(\alpha_1)$. 
By Lemma \ref{lem-even}, $\phi(\alpha_2)$ is a separating curve in $S$. 
The two BPs $\phi(\alpha_1)$ and $\phi(\alpha_3)$ are contained in the same component of $S_{\phi(\alpha_2)}$ by Lemma \ref{lem-bp}, while $\alpha_1$ and $\alpha_3$ are contained in distinct components of $S_{\alpha_2}$. 
This contradicts Proposition \ref{prop-chi-pre}.
\end{proof}

We proved that $\phi$ preserves separating curves in $S$.
By Theorem \ref{thm-cs}, $\phi$ is induced by an element of $\mod^*(S)$.
The proof of Theorem \ref{thm-other} is completed.
\end{proof}


\section{Proof of Theorems \ref{thm-comm} and \ref{thm-coh}}\label{sec-coh}

\begin{proof}[Proof of Theorem \ref{thm-comm}]
Let $S=S_{g, p}$ be a surface with $g\geq 1$ and $|\chi(S)|\geq 4$.
We put $G=\mod^*(S)$ and pick a subgroup $\Gamma$ of $\mod^*(S)$ with $[\calk(S): \Gamma \cap \calk(S)]<\infty$ and $[\Gamma :\Gamma \cap \cali(S)]<\infty$.
Note that $\Gamma$ contains a finite index subgroup of $\calk(S)$ and that there exists a finite index subgroup of $\Gamma$ contained in $\cali(S)$.
Let
\[{\bf i}\colon \comm_G(\Gamma)\rightarrow \comm(\Gamma)\]
be the natural homomorphism.
If $\gamma$ is an element of $\comm_G(\Gamma)$ with ${\bf i}(\gamma)$ neutral, then there exists a finite index subgroup of $\Gamma$ such that any element in it commutes $\gamma$.
It follows that for any $a\in V_s(S)$, there exists a non-zero integer $n$ with $\gamma t_a^n\gamma^{-1}=t_a^n$, and thus we have $\gamma a =a$.
For each non-separating curve $b$ in $S$, we can find two separating curves $a_1$, $a_2$ in $S$ such that they are disjoint from $b$ and fill $S_b$.
Since $\gamma$ fixes $a_1$ and $a_2$, it also fixes $b$.
The element $\gamma$ thus fixes any element of $V(S)$ and is neutral.
Injectivity of ${\bf i}$ follows.

To prove surjectivity of ${\bf i}$, we may assume that $\Gamma$ is contained in $\cali(S)$.
Let $\Gamma_1$ and $\Gamma_2$ be finite index subgroups of $\Gamma$, and let $f\colon \Gamma_1\rightarrow \Gamma_2$ be an isomorphism.
Since $\Gamma_1\cap \calk(S)$ is a finite index subgroup of $\calk(S)$, there exists an element $\gamma_0$ of $G$ with the equality $f(\gamma)=\gamma_0\gamma \gamma_0^{-1}$ for any $\gamma \in \Gamma_1\cap \calk(S)$ by Theorem \ref{thm-inner}.

We claim that the equality $f(\gamma)=\gamma_0\gamma \gamma_0^{-1}$ holds for any $\gamma \in \Gamma_1$.
This implies surjectivity of {\bf i}.
Pick $\gamma \in \Gamma_1$.
For any $a\in V_s(S)$, we have
\begin{align*}
f(\langle t_{\gamma a}\rangle \cap \Gamma_1)&=f(\gamma (\langle t_a\rangle \cap \Gamma_1)\gamma^{-1})<f(\gamma)(\langle \gamma_0t_a\gamma_0^{-1}\rangle \cap \Gamma_2)f(\gamma)^{-1}\\
&=\langle t_{f(\gamma)\gamma_0 a}\rangle \cap \Gamma_2,
\end{align*}
where for each $x\in G$, $\langle x\rangle$ denotes the group generated by $x$.
We thus have $\gamma_0\gamma a=f(\gamma)\gamma_0 a$.
It follows that $\gamma^{-1}\gamma_0^{-1}f(\gamma)\gamma_0$ fixes any element of $V_s(S)$.
The argument in the end of the first paragraph of the proof shows that $\gamma^{-1}\gamma_0^{-1}f(\gamma)\gamma_0$ is neutral.
\end{proof}

To prove Theorem \ref{thm-coh}, we need the following two lemmas.

\begin{lem}\label{lem-coh}
Let $N$ be a group with a finite chain of subgroups, $\{ e\}=N_0<N_1<\cdots <N_m=N$, such that for each $j=1,\ldots, m$,
\begin{itemize}
\item $N_{j-1}$ is a normal subgroup of $N_j$; and
\item $N_j/N_{j-1}$ is finitely generated and abelian.
\end{itemize}
Let $H$ be a subgroup of $N$.
If $f$ is an automorphism of $N$ with $f(H)<H$ and $f(N_j)=N_j$ for each $j=0, 1,\ldots, m$, then we have $f(H)=H$.
\end{lem}

\begin{proof}
We put $H_j=H\cap N_j$ for $j=0, 1,\ldots, m$.
We prove the equality $f(H_j)=H_j$ by induction on $j$.
When $j=0$, it is obvious.
We assume $j\geq 1$.
Put $M=N_j/N_{j-1}$ and let $\phi$ denote the automorphism of $M$ induced by $f$.
Let $K$ denote the subgroup $H_j/H_{j-1}$ of $M$.
Since $M$ is abelian, $K$ is a normal subgroup of $M$.
By assumption, we have $\phi(K)<K$.
It follows that $\phi$ induces a surjective homomorphism from $M/K$ onto itself, which is injective because $M/K$ is finitely generated and abelian.
The equality $\phi(K)=K$ thus holds.
On the other hand, by the hypothesis of the induction, the equality $f(H_{j-1})=H_{j-1}$ holds.
We therefore obtain the equality $f(H_j)=H_j$.
\end{proof}

\begin{lem}\label{lem-chain}
Let $S=S_{1, p}$ be a surface with $p\geq 1$.
Then there exists a chain of normal subgroups of $\pmod(S)$,
\[\calk(S)=N_0<N_1<\cdots <N_{p-2}<N_{p-1}=\cali(S),\]
such that $N_j/N_{j-1}$ is finitely generated and abelian for any $j=1,\ldots, p-1$.
\end{lem}

\begin{proof}
We prove the lemma by induction on $p$.
If $p=1$, then $\cali(S)$ is trivial, and the lemma follows.
Assume $p\geq 2$.
Pick a boundary component $\partial$ of $S$, and let $\bar{S}$ denote the surface obtained by attaching a disk to $\partial$.
Let
\[1\rightarrow \pi_1(\bar{S})\stackrel{\iota}{\rightarrow}\pmod(S)\stackrel{\theta}{\rightarrow}\pmod(\bar{S})\rightarrow 1\]
be the associated Birman exact sequence (see Chapter 4 in \cite{birman}).
We then have
\[\theta(\cali(S))=\cali(\bar{S}),\quad \theta(\calk(S))=\calk(\bar{S}),\quad \iota(\pi_1(\bar{S})) <\cali(S).\]
We define $H$ as the subgroup of $\cali(S)$ generated by $\calk(S)$ and $\iota(\pi_1(\bar{S}))$, which is a normal subgroup of $\pmod(S)$.
By the hypothesis of the induction, there exists a chain of normal subgroups of $\pmod(\bar{S})$,
\[\calk(\bar{S})=\bar{N}_0<\bar{N}_1<\cdots <\bar{N}_{p-2}=\cali(\bar{S}),\]
such that $\bar{N}_j/\bar{N}_{j-1}$ is finitely generated and abelian for any $j=1,\ldots, p-2$.
Setting $N_j=\theta^{-1}(\bar{N}_{j-1})$ for each $j=1,\ldots, p-1$, we obtain the chain of normal subgroups of $\pmod(S)$,
\[H=N_1<N_2<\cdots <N_{p-1}=\cali(S)\]
with $N_j/N_{j-1}$ isomorphic to $\bar{N}_{j-1}/\bar{N}_{j-2}$ for each $j=2,\ldots, p-1$.

We claim that $H/\calk(S)$ is finitely generated and abelian.
This claim completes the induction.
Since we have a homomorphism from $\pi_1(\bar{S})$ onto $H/\calk(S)$, the group $H/\calk(S)$ is finitely generated.
Choose generators $a$, $b$, $c_1,\ldots, c_{p-1}$ of $\pi_1(\bar{S})$ such that each of them is a simple loop in $\bar{S}$; any two of them intersect only at the base point; $a$ and $b$ are non-separating in $\bar{S}$; and each $c_j$ surrounds a single component of $\partial \bar{S}$.
By the definition of $\iota$, any $\iota(c_j)$ belongs to $\calk(S)$ because $\iota(c_j)$ is the Dehn twist about a p-curve in $S$.
The commutator $[a, b]$ represents a separating simple loop cutting off a handle from $\bar{S}$.
By the definition of $\iota$, $\iota([a, b])$ is written as $t_{\alpha}t_{\beta}^{-1}$ for some distinct elements $\alpha$, $\beta$ of $V_s(S)$ with $i(\alpha, \beta)=0$.
It follows that $\iota([a, b])$ belongs to $\calk(S)$.
The group $H/\calk(S)$ is therefore abelian.
\end{proof}

\begin{proof}[Proof of Theorem \ref{thm-coh}]
Let $S=S_{g, p}$ be a surface with either $g\geq 3$ and $p\leq 1$ or $g=1$ and $p\geq 4$.
Let $\Gamma$ be a subgroup of $\mod^*(S)$ with $[\calk(S): \Gamma \cap \calk(S)]<\infty$ and $[\Gamma :\Gamma \cap \cali(S)]<\infty$.
Let $f\colon \Gamma \rightarrow \Gamma$ be an injective homomorphism.
We put $\Gamma'=f^{-1}(\Gamma \cap \cali(S))\cap \calk(S)$, which is a finite index subgroup of $\calk(S)$.
Applying Theorem \ref{thm-inner} to the restriction of $f$ to $\Gamma'$, we obtain an element $\gamma$ of $\mod^*(S)$ with the equality $f(x)=\gamma x\gamma^{-1}$ for any $x\in \Gamma'$.
As in the final paragraph of the proof of Theorem \ref{thm-comm}, we can show that this equality also holds for any $x\in \Gamma$.
The inclusion $\gamma\Gamma \gamma^{-1}<\Gamma$ thus holds.
We have to show that the equality indeed holds.

For any positive interger $n$, we have $\gamma^n\Gamma \gamma^{-n}<\gamma^{n-1}\Gamma \gamma^{-n+1}<\Gamma$.
It follows that if there exists a positive integer $n$ with $\gamma^n\Gamma \gamma^{-n}=\Gamma$, then the desired equality is obtained.
We may therefore assume that $\gamma$ is an element of $\pmod(S)$.

We set
\[\Gamma_0=\Gamma \cap \calk(S),\quad \Gamma_1=\Gamma \cap \cali(S),\quad \cal{Q}=\pmod(S)/\calk(S)\]
and define $\imath \colon \pmod(S)\rightarrow \cal{Q}$ as the canonical quotient map.
For each $j=0, 1$, the inclusion $\gamma\Gamma_j\gamma^{-1}<\Gamma_j$ then holds.
Since we have
\[[\calk(S):\gamma\Gamma_0\gamma^{-1}]=[\gamma\calk(S)\gamma^{-1}:\gamma\Gamma_0\gamma^{-1}]=[\calk(S):\Gamma_0]<\infty,\]
we obtain $[\Gamma_0: \gamma\Gamma_0\gamma^{-1}]=1$ and thus $\gamma\Gamma_0\gamma^{-1}=\Gamma_0$.

We now assume $g\geq 3$ and $p\leq 1$.
It is known that $\imath(\cali(S))=\cali(S)/\calk(S)$ is a finitely generated and abelian group, due to Johnson (see Theorems 5 and 6 in \cite{johnson2}).
Since the conjugation by $\imath(\gamma)$ induces an automorphism of $\imath(\cali(S))$, the inclusion $\imath(\gamma)\imath(\Gamma_1)\imath(\gamma)^{-1}<\imath(\Gamma_1)$ implies the equality $\imath(\gamma)\imath(\Gamma_1)\imath(\gamma)^{-1}=\imath(\Gamma_1)$.
Combining the equality $\gamma \Gamma_0\gamma^{-1}=\Gamma_0$, we obtain $\gamma\Gamma_1\gamma^{-1}=\Gamma_1$.
Since we have
\[[\Gamma :\Gamma_1]=[\gamma\Gamma \gamma^{-1}: \gamma\Gamma_1\gamma^{-1}]=[\gamma\Gamma \gamma^{-1}:\Gamma_1]<\infty,\]
we obtain $[\Gamma :\gamma\Gamma \gamma^{-1}]=1$ and thus $\gamma\Gamma \gamma^{-1}=\Gamma$.

We next assume $g=1$ and $p\geq 4$.
By Lemma \ref{lem-chain}, we have a chain of normal subgroups of $\pmod(S)$,
\[\calk(S)=N_0<N_1<\cdots <N_{p-2}<N_{p-1}=\cali(S),\]
such that $N_j/N_{j-1}$ is finitely generated and abelian for any $j=1,\ldots, p-1$.
By applying Lemma \ref{lem-coh} to the chain of normal subgroups of $\cal{Q}$,
\[\{ e\} =\imath(N_0)<\imath(N_1)<\cdots <\imath(N_{p-2})<\imath(N_{p-1})=\imath(\cali(S)),\]
and to the inclusion $\imath(\gamma)\imath(\Gamma_1)\imath(\gamma)^{-1}<\imath(\Gamma_1)$, we obtain $\imath(\gamma)\imath(\Gamma_1)\imath(\gamma)^{-1}=\imath(\Gamma_1)$.
Following argument in the previous paragraph, we obtain $\gamma\Gamma \gamma^{-1}=\Gamma$.
\end{proof}


\end{document}